\ifx \labelsONmargin\undefined

  \ifx\RequirePackage\undefined
    \expandafter\ifx\csname amsart.sty\endcsname\relax
        \documentstyle[12pt,amssymb,amscd]{amsart}
    \fi

    \def\atSign{@@}

    \def\mathbb{\Bbb}
    \def\mathfrak{\frak}
    \def\mathbf{\bold}
    \ifx\boldsymbol\undefined	
      \def\boldsymbol#1{{\bold #1}}
    \fi

    \def\mathbit{\boldsymbol}

    \makeatletter
    \newenvironment{proof}{%
         \@ifnextchar[{%
                       \expandafter\let\expandafter\end@proof
                         \csname endpf*\endcsname
                         \my@proof
                      }{\let\end@proof\endpf\pf}%
        }{\end@proof}
    \def\my@proof[#1]{\@nameuse{pf*}{#1}}
    \makeatother

    \def\xrightarrow[#1]#2{@>{#2}>{#1}>}
    \def\xleftarrow[#1]#2{@<{#2}<{#1}<}

    \def\providecommand#1{\def#1}
    \def\emph#1{{\em #1}}
    \def\textbf#1{{\bf #1}}

    \def\mathring{\overset{\,\,{}_\circ}}
  \else
      \ifx\usepackage\RequirePackage
      \else
        \documentclass[12pt]{amsart}
        \usepackage{amssymb,amscd}
     \usepackage{xcolor}
     \usepackage[normalem]{ulem}
     \usepackage{cancel}
     \usepackage{amsmath}   
      \fi
      \ifx \mathring\undefined		
        \DeclareMathAccent{\mathring}{\mathalpha}{operators}{"17}
      \fi
 
      \long\def\FAKEendPROOF{\endtrivlist}
      \ifx\endproof\FAKEendPROOF
	  \def\endproof{\qed\endtrivlist}
      \fi

      \ifx\mathbit\undefined	
        \DeclareMathAlphabet{\mathbit}{OML}{cmm}{b}{it}
      \fi

      \def\atSign{@}

      \advance\oddsidemargin -0.1\textwidth
      \evensidemargin\oddsidemargin
      \def\Sb#1\endSb{_{\substack{#1}}}
      \def\Sp#1\endSp{^{\substack{#1}}}

  \fi
\makeatletter
\@ifundefined{DeclareFontShape} 
     {\@ifundefined{selectfont}
             {
                \message{We need AmS-LaTeX, this version of LaTeX does not 
                        support it}\@@end
             }{
                \def\mathcal{\cal}
                
                \def\pcyr{%
                        \def\default@family{UWCyr}%
                        \let\oldSl@\sl
                        \def\sl{\def\default@shape{it}\oldSl@}%
                        \cyracc
                        \language\Russian\family{UWCyr}\selectfont
                }
             }}{
                \DeclareFontEncoding{OT2}{}{\noaccents@}
                \DeclareFontSubstitution{OT2}{cmr}{m}{n}
                \DeclareFontFamily{OT2}{cmr}{\hyphenchar\font45 }
                \DeclareFontShape{OT2}{cmr}{m}{n}{%
                     <5><6><7><8><9><10>gen*wncyr %
                     <10.95><12><14.4><17.28><20.74><24.88> wncyr10 %
                }{}
                \DeclareFontShape{OT2}{cmr}{m}{it}{%
                     <5><6><7><8><9><10> gen * wncyi%
                     <10.95><12><14.4><17.28><20.74><24.88> wncyi10%
                }{}
                \DeclareFontShape{OT2}{cmr}{bx}{n}{%
                     <5><6><7><8><9><10> gen * wncyb%
                     <10.95><12><14.4><17.28><20.74><24.88> wncyb10%
                }{}
                \DeclareFontShape{OT2}{cmr}{m}{sl}{%
                     <-> ssub * cmr/m/it%
                }{}
                \DeclareFontShape{OT2}{cmr}{m}{sc}{%
                     <5><6><7><8><9><10>%
                     <10.95><12><14.4><17.28><20.74><24.88> wncysc10%
                }{}
                \DeclareFontFamily{OT2}{cmss}{\hyphenchar\font45 }
                \DeclareFontShape{OT2}{cmss}{m}{n}{%
                     <8><9><10> gen * wncyss%
                     <10.95><12><14.4><17.28><20.74><24.88> wncyss10%
                }{}
                
                \def\cyrencodingdefault{OT2}
                \def\pcyr{%
                        \cyracc
                        \let\encodingdefault\cyrencodingdefault
                        \language\Russian\fontencoding{OT2}\selectfont
                }
             }   

\@ifundefined{theorembodyfont} 
     {
        \def\theorembodyfont#1{\relax}
        \makeatletter
          \let\@@th@plain\th@plain
          \def\th@plain{ \@@th@plain \slshape }
        \makeatother
        \let\normalshape\relax
     }{}   

\ifx\cprime\undefined
     \def\cprime{$'$}
\fi
\makeatletter
  \def\@sect@my#1#2#3#4#5#6[#7]#8{%
\ifnum #2>\c@secnumdepth
   \let\@svsec\@empty
 \else
   \refstepcounter{#1}%
\edef\@svsec{\ifnum#2<\@m
             \@ifundefined{#1name}{}{\csname #1name\endcsname\ }\fi
\noexpand\rom{\csname the#1\endcsname.}\enspace}\fi
 \@tempskipa #5\relax
 \ifdim \@tempskipa>\z@ 
   \begingroup #6\relax
   \@hangfrom{\hskip #3\relax\@svsec}{\interlinepenalty\@M #8\par}%
   \endgroup
   \if@article\else\csname #1mark\endcsname{%
        \ifnum \c@secnumdepth >#2\relax\csname the#1\endcsname. \fi#7}\fi
\ifnum#2>\@m \else
       \let\@tempf\\ \def\\{\protect\\}\addcontentsline{toc}{#1}%
{\ifnum #2>\c@secnumdepth \else
             \protect\numberline{%
               \ifnum#2<\@m
               \@ifundefined{#1name}{}{\csname #1name\endcsname\ }\fi
               \csname the#1\endcsname.}\fi
           #8}\let\\\@tempf
     \fi
 \else
  \def\@svsechd{#6\hskip #3\@svsec
    \@ifnotempty{#8}{\ignorespaces#8\unskip
       \ifnum\spacefactor<1001.\fi}%
        \ifnum#2>\@m \else
          \let\@tempf\\ \def\\{\protect\\}\addcontentsline{toc}{#1}%
            {\ifnum #2>\c@secnumdepth \else
              \protect\numberline{%
                \ifnum#2<\@m
                \@ifundefined{#1name}{}{\csname #1name\endcsname\ }\fi
                \csname the#1\endcsname.}\fi
             #8}\let\\\@tempf\fi}%
 \fi
\@xsect{#5}}
  \ifx\RequirePackage\undefined
  \let\@sect\@sect@my             
  \fi
  \ifx\RequirePackage\undefined	
  \def\th@remark@my{\theorempreskipamount6\p@\@plus6\p@
    \theorempostskipamount\theorempreskipamount
    \def\theorem@headerfont{\it}\normalshape}
    \let\th@remark\th@remark@my
  \else				
    \let\o@@remark\th@remark

    \ifx\theorempostskipamount\undefined		
    \else						
      \def\th@remark{\o@@remark
	\ifdim\theorempostskipamount < 2pt\relax
	  \theorempostskipamount\theorempreskipamount
	     \multiply\theorempostskipamount\tw@
	     \divide\theorempostskipamount\thr@@
	\fi
      }
    \fi
  \fi
\let\myLabel\@gobble
\def\labelsONmargin{\@mparswitchfalse\def\myLabel##1{\@bsphack\marginpar
                                  {\normalshape\tiny\rm Label ##1}\@esphack}}
\ifx \url\undefined
  \def\url#1{{\tt #1}}%
\fi
\def\cyracc{\def\u##1{
                \if \i##1\char"1A%
                \else \if I##1\char"12%
                \else \accent"24 ##1\fi\fi }%
\def\"##1{\if e##1{\char"1B}%
                \else \if E##1{\char"13}%
                \else \accent"7F ##1\fi\fi }%
\def\9##1{\if##1z\char"19 
\else\if##1Z\char"11 
\else\if##1E\char"03 
\else\if##1e\char"0B 
\else\if##1u\char"18 
\else\if##1U\char"10 
\else\if##1A\char"17 
\else\if##1a\char"1F 
\else\if##1p\char"7E 
\else\if##1P\char"5E 
\else\if##1Q\char"5F 
\else\if##1q\char"7F 
\else\if##1i\char"1A 
\else\if##1I\char"12 
\else\if##1N\char"7D 
\fi
\fi
\fi
\fi
\fi
\fi
\fi
\fi
\fi
\fi
\fi
\fi
\fi
\fi
\fi
}%
\def\cydot{{\kern0pt}}}%
\def\cydot{$\cdot$}

\ifx\Russian\undefined
        \def\Russian{0\relax
    \message{Don't know the hyphenation rules for Russian^^J
                        Please do INITeX with `input  russhyph' in the 
                        command line}%
                \gdef\Russian{0\relax}%
        }
\fi
\ifx\RequirePackage\undefined			
  \def\@putname#1#2#3#4{\def\@@ref{#3}\let\old@bf\bf
        \def\bf##1{\old@bf\if?\noexpand##1?{#4}\else##1\fi}%
	#1{#2}%
        \let\bf\old@bf}
\else						
  \def\@putname#1#2#3#4{\def\@@ref{#3}\let\old@bf\bf	
	\let\old@reset@font\reset@font			
        \def\bf##1{\old@bf\if?\noexpand##1?{#4}\else##1\fi}%
	\def\reset@font##1##2{\old@reset@font##1\if?\noexpand##2?{#4}\else##2\fi}#1{#2}%
        \let\bf\old@bf\let\reset@font\old@reset@font}
\fi
\let\my@ref=\ref
\def\ref#1{\@putname\my@ref{#1}{#1}{\tiny\rm\@@ref}}
\let\my@pageref=\pageref
\def\pageref#1{\@putname\my@pageref{#1}{#1}{\tiny\rm\@@ref}}
\let\my@cite=\cite
\def\cite#1{\@putname\my@cite{#1}{\@citeb}{\tiny\rm\@@ref}}
\makeatother
\fi
\relax 
\theoremstyle{plain} 
\theorembodyfont{\sl}

\numberwithin{equation}{section}
\theorembodyfont{\rm}
\theoremstyle{definition}

\newtheorem{definition}{Definition}[section]

\newtheorem{example}[definition]{Example}

\theoremstyle{remark}

\newtheorem{remark}[definition]{Remark} 

\theoremstyle{plain} 
\theorembodyfont{\sl}

\newtheorem{theorem}[definition]{Theorem}
\newtheorem{lemma}[definition]{Lemma}
\newtheorem{corollary}[definition]{Corollary}
\newtheorem{proposition}[definition]{Proposition}

\renewcommand{\gg}{\mathfrak{g}}

\newcommand{\hh}{\mathfrak{h}}

\newcommand{\cW}{\mathcal{W}}
\renewcommand{\sp}{\mathfrak{sp}}
\newcommand{\so}{\mathfrak{o}}
\renewcommand{\sl}{\mathfrak{sl}}
\newcommand{\osp}{\mathfrak{osp}}
\newcommand{\ggl}{\mathfrak{gl}}
\renewcommand{\dim}{\mathrm{dim}}

\newcommand{\CC}{\mathbb{C}}
\newcommand{\ZZ}{\mathbb{Z}}
\newcommand{\limarr}{\underrightarrow{\mathrm{lim}} \;}

\newcommand{\Ann}{\mathrm{Ann}}

\newcommand{\ad}{\mathrm{ad}}
\newcommand{\Hom}{\mathrm{Hom}}
\newcommand{\Ext}{\mathrm{Ext}}
\newcommand{\End}{\mathrm{End}}

\newcommand{\supp}{\mathrm{supp}}

\newcommand{\C}{\mathbb{C}}
\newcommand{\Z}{\mathbb{Z}}
\newcommand{\gs}{\mathfrak{s}}
\newcommand{\gh}{\mathfrak{h}}

\def\cplus{\hbox{$\subset${\raise1.05pt\hbox{\kern -0.55em
${\scriptscriptstyle +}$}}\ }}

\def\gb{\mathfrak{b}}

\def\gg{\mathfrak{g}}
\def\gh{\mathfrak{h}}

\def\gk{\mathfrak{k}}
\def\gl{\mathfrak{l}}

\def\go{\mathfrak{o}}
\def\gp{\mathfrak{p}}
\def\gq{\mathfrak{q}}

\def\gs{\mathfrak{s}}

\begin{document}
\bibliographystyle{amsplain}
\relax

\title[Bounded weight modules]{Bounded weight modules for basic classical Lie superalgebras at infinity}
\author{Dimitar Grantcharov, Ivan Penkov and Vera Serganova}

\address{Department of Mathematics,
 University of Texas at Arlington, Arlington, TX 76019-0408, 
USA}

\email{grandim\atSign{}uta.edu }

\address{Jacobs University Bremen, 
 Campus Ring 1,
28759 Bremen, Germany}

\email{i.penkov\atSign{}jacobs-university.de}

\address{Deptartment of Mathematics, University of California at Berkeley,
Berkeley, CA 94720, USA}

\email{serganov\atSign{}math.berkeley.edu}

\maketitle

\begin{abstract}
We classify simple bounded weight modules over the complex simple Lie superalgebras $\sl(\infty |\infty)$ and $\osp (m | 2n)$, when at least one of $m$ and $n$ equals $\infty$. For  $\osp (m | 2n)$ such modules are of spinor-oscillator type, i.e., they combine into one the known classes of spinor $\go (m)$-modules and oscillator-type $\sp (2n)$-modules. In addition, we characterize the category of bounded weight modules over $\osp (m | 2n)$ (under the assumption $\dim \, \osp (m | 2n) = \infty$) by reducing its study to already known categories of representations of $\sp (2n)$, where $n$ possibly equals $\infty$. When classifying simple bounded weight $\sl(\infty |\infty)$-modules, we prove that every such module is integrable over one of the two infinite-dimensional ideals of the Lie algebra  $\sl(\infty |\infty)_{\bar{0}}$. We finish the paper by establishing some first  facts about the category of bounded weight $\sl(\infty |\infty)$-modules.
\end{abstract}

{\small
\medskip\noindent 2020 MSC: Primary 17B65, 17B10 \\
\noindent Keywords and phrases: direct limit Lie superalgebra, Clifford superalgebra, Weyl superalgebra, 
weight module, annihilator.}

\section*{Introduction}
The representation theory of the three simple infinite-dimensional finitary complex Lie algebras  $\sl (\infty)$, $\mathfrak o (\infty)$, and $\sp (\infty)$ has made notable progress in the last three decades, see for instance  \cite{DPS},  \cite{DP}, \cite{PSer1}, \cite{PSer2}, \cite{PStyr}, \cite{SS}. For a  summary of highlights of this theory  see  \cite{PH}. The theory of representations of the super-counterparts of the Lie algebras $\sl (\infty)$, $\mathfrak o (\infty)$, and $\sp (\infty)$ is still much less developed. For a finite-dimensional Lie superalgebra $\mathfrak k$, the category of all representations of $\mathfrak k$ is almost never equivalent to the category of all representations of the Lie algebra $\mathfrak k_{\bar{0}}$, the even part  of  $\mathfrak k$. However, in that case  there is a general result claiming that 
a category of representations of $\mathfrak k$ with fixed strongly typical central character is equivalent to a corresponding category of representations of $\mathfrak k_{\bar{0}}$.

 This result does not provide a clear guideline for the case of Lie superalgebras of infinite rank  since  the center of the enveloping algebra of  Lie superalgebras like $\sl (\infty | \infty)$ or $\osp (\infty | \infty)$ is trivial. Nevertheless, in the study of  reasonably small categories of representations over the Lie superalgebras  $\sl (\infty | \infty)$ and $\osp (\infty | \infty)$, one may rely on different intuition and obtain results not necessarily following the above pattern. For instance, in \cite{S} it is shown that the category of tensor modules  over the Lie superalgebra  $\osp (\infty | \infty)$ (respectively, over  $\sl (\infty | \infty)$) is  equivalent to the categories of tensor modules  over each of the Lie algebras $\mathfrak o (\infty)$ and $\sp (\infty)$ (respectively, over $\sl (\infty)$).  A somewhat similar phenomenon can be seen in the paper \cite{CP}, where it is proved that the categories of integrable bounded weight  modules over various Lie superalgebras like $\sl (\infty | \infty)$ or $\osp (\infty | \infty)$ are semisimple.

In the present paper we study the categories of arbitrary (i.e., not necessarily  integrable) bounded weight modules over the complex Lie superalgebras $\osp (m | 2n)$,  where at least one of $m$ or $n$ equals $\infty$, and over the Lie superalgebra $\sl (\infty | \infty)$. Before describing our results we should recall that for the infinite-dimensional Lie algebras $\mathfrak{sl} (\infty)$, $\mathfrak o (\infty)$, $\mathfrak{sp}(\infty)$ simple bounded weight modules have been classified in \cite{GrP} and their structure has been further studied in \cite{C}.

Our first main result claims that any simple bounded weight module over an  infinite-dimensional Lie superalgebra $\osp (m | 2n)$ has  just  length two (or one for a trivial module) over the Lie algebra  $\osp (m | 2n)_{\bar{0}} = \mathfrak o (m ) \oplus \sp( n)$. Moreover, such a module  (unless it is a natural or trivial module) is determined by a pair $(S,N)$, where $S$ is a spinor $\mathfrak o (m)$-module and $N$ is a $\sp (2n)$-module of oscillator type, i.e., a close relative of the oscillator representations of  $\sp (2n)$. (The notions of spinor $ \mathfrak o (m)$-modules and oscillator-type $ \mathfrak {sp} (2n)$-modules  make sense also for $m=\infty$ and $n=\infty$ due to the results of \cite{GrP}.)  This spectacular fact allows us to identify simple bounded weight $\osp (m | 2n)$-modules, other than trivial and natural modules, as modules of ``spinor-oscillator type''. The latter class of modules of $\osp(m|2n)$  glues spinor and oscillator-type modules together, and is the ultimate super-symmetric version of both   spinor $\mathfrak o (m)$-modules  and oscillator-type $\sp (2n)$-modules. 

The classification of simple bounded weight $\sl (\infty | \infty)$-modules is also very interesting and constitutes our second main result. In particular, we show that every such module is integrable and semisimple with respect to a simple ideal of $\sl (\infty | \infty)_{\bar{0}} \simeq \left(\sl (\infty) \oplus \sl(\infty) \right) \cplus \mathbb C$, and this nicely resembles the answer for the case of $\osp (\infty | \infty)$ where a  bounded weight $\osp (\infty | \infty)$-module is necessarily integrable and semisimple  as an $\mathfrak o (\infty)$-module.

Our main method of classification is reduction to weight modules of Weyl and Clifford superalgebras of infinitely many variables. We denote these superalgebras respectively by $D(\infty |\infty)$ and $Cl(\infty |\infty)$. There are natural homomorphisms $U(\osp (\infty | \infty)) \to Cl(\infty |\infty)$ and $U(\gs \gl (\infty |\infty)) \to D(\infty |\infty)$, see Section 2.2. One of our central ideas is that, with the exception of Schur powers of the natural and conatural representations (for $\osp (\infty | \infty)$ this exception applies only to the trivial and natural representations), all simple bounded weight $\osp (\infty |\infty)$-  or $\sl (\infty | \infty)$-modules are annihilated by the kernel of the respective homomorphism. This facilitates a reduction of the study of simple bounded weight $\osp (\infty | \infty)$-  and $\sl (\infty | \infty)$-modules, as well as of the respective categories of bounded weight modules, to the study of weight modules of the associative superalgebras  $Cl(\infty |\infty)$ and $D(\infty |\infty)$ and their relevant subalgebras. The above method applies also to the case of $\osp (m|2n)$ where $m$ or $n$ is finite, and to $\sl (\infty | n)$ for $n \in \Z_{>0}$ as well. 

Here is a brief description of the content of the paper. Section 1 is devoted to preliminaries. In Section 2 we undertake a study of the categories of weight modules over Clifford and Weyl  superalgebras. In particular, we establish that any such simple  module is multiplicity free.  In Sections 3 and 4 we apply the above results to the case of $\osp (m|2n)$  where at least one of $m$ and $n$ equals infinity. We show that any simple non-integrable bounded weight  $\osp (m|2n)$-module is a spinor-oscillator module. Moreover, we prove that the category of spinor-oscillator representations is equivalent to the category of multiplicity free non-integrable weight modules over the Lie algebra  $\osp (m|2n)_{\bar{0}} = \go (m) \oplus \sp (2n)$.

The case of $\sl (\infty |\infty)$ is discussed in Section 5.  Here we give a classification of the simple bounded weight $\sl (\infty |\infty)$-representations and make a first step towards  understanding  the category of such representations. A deeper study of this category should be a separate project.

\medskip
{\bf Acknowledgment.} We thank Lucas Calixto for reading a preliminary draft of the paper. IP and VS acknowledge the outstanding hospitality of Mathematisches Forschungsinstitut Oberwolfach where this paper was almost completed.
DG was supported in part by Simons Collaboration Grant 855678. IP was supported in part by DFG grant PE 980/8-1.  VS was supported in part by
NSF grant 2001191. 
   \section{Preliminaries}\label{prelim}
   
   The base field is $\C$. By $S_n$ we denote the symmetric group on $n$ letters. A \emph{superspace} is a $\mathbb Z_2$-graded vector space where $\mathbb Z_2 := \mathbb Z/2\mathbb Z$, and a \emph{superalgebra} is a $\mathbb Z_2$-graded algebra. We use the indices $\bar{0}$ and $\bar{1}$ 
  to indicate $\mathbb Z_2$-gradings.  
  A \emph{purely even}  (respectively, \emph{purely odd}) superspace is a superspace $V$ such that   $V = V_{\bar{0}}$ (resp., $V = V_{\bar{1}}$). By $\Pi$ we denote the parity change functor on  superspaces:  $(\Pi V)_{\bar{0}} = V_{\bar{1}} $, $(\Pi V)_{\bar{1}} = V_{\bar{0}}$. If $V = V_{\bar{0}} \oplus  V_{\bar{1}}$ is a superspace, then the \emph{dual} superspace equals $V_{\bar{0}}^* \oplus  V_{\bar{1}}^*$, where  $V_{\bar{0}}^* = \Hom (V_{\bar{0}}, \mathbb C)$, $V_{\bar{1}}^* = \Pi \Hom (\Pi V_{\bar{1}}, \mathbb C)$ and $\Hom$ stands here for homomorphisms of purely even spaces.

We write $S^k V$  and  $\Lambda^k V$  for the $k$th  symmetric and exterior powers for a superspace $V$. If $W$ is a superspace of parity $p \in \mathbb Z_2 $ (i.e., $W = W_{\bar{0}}$ for $p=\bar{0}$ and $W = W_{\bar{1}}$ for $p=\bar{1}$), then $S^k W$  (respectively, $\Lambda^k W$) is a superspace of parity $kp \in \mathbb Z_2 $ (respectively, $kp + \bar{1} \in \mathbb Z_2 $). For a general superspace $V = V_{\bar{0}} \oplus V_{\bar{1}}$ we have
      $$
   S^kV = \bigoplus_{i+j = k} {S}^{i}V_{\bar{0}} \otimes \Lambda^{j} V_{\bar{1}}, \;   \Lambda^kV = \bigoplus_{i+j = k} \Lambda^{i}V_{\bar{0}} \otimes S^{j} V_{\bar{1}}.
   $$
  An \emph{ even  symmetric} (respectively, \emph{even antisymmetric}) bilinear form on a superspace $V$ is a parity-preserving linear operator $S^2V \to \C$ (respectively, $\Lambda^2V \to \C$).

  In this paper we work with the Lie superalgebras $\mathfrak{gl}(a|b)$, $\sl (a|b)$, $\osp (2a|2b)$, $\osp (2a+1|2b)$, where $a, b \in \mathbb Z_{\geq 0} \cup \{ \infty\}$. Their  \emph{defining representation} is the  simple module of respective dimension $(a|b)$,  $(a|b)$, $(2a|2b)$, $(2a+1|2b)$. In what follows we use the term defining representation more loosely to include also the defining representation with changed parity. The Lie superalgebras $\mathfrak{gl}(a|b)$, $\sl (a|b)$, $\osp (2a|2b)$, $\osp (2a+1|2b)$ can be equipped with a fixed even symmetric invariant form  $(\cdot, \cdot)$. All homomorphisms of superalgebras are assumed to preserve the $\mathbb Z_2$-grading.  All modules over purely even  associative algebras or Lie algebras are assumed to be purely even unless otherwise stated.

  We assume that  Cartan subalgebras of the Lie superalgebras considered are fixed, and use standard notation for the roots.  Note that these Cartan subalgebras are purely even and all root spaces are either purely even or purely odd. Therefore the roots are designated as even or odd.    Concretely, the even roots of  $\mathfrak{gl}(a|b)$ and  $\sl (a|b)$ are 
$\varepsilon_i - \varepsilon_{k}, \delta_j - \delta_l$, while the odd roots are $\pm (\varepsilon_i - \delta_j)$, where $1\leq i \neq k \leq a, 1 \leq j \neq l \leq b 
$. The even roots of $\mathfrak{osp}(2a|2b)$ are $\pm ( \varepsilon_i \pm  \varepsilon_{k}), \pm (\delta_j \pm  \delta_l), \pm 2 \delta_j$, and the odd roots are $\pm (\varepsilon_i - \delta_j)$. For $\mathfrak{osp}(2a+1|2b)$ we have in addition the even roots $ \pm  \varepsilon_i $ and the odd roots  $\pm \delta_j$.

We should point out that for $a=\infty$  the Lie superalgebras $\osp(2a+1|2b)$ and $\osp(2a|2b)$ are isomorphic, and the difference in root systems is the result of different choices of Cartan subalgebras. A less brief discussion of the Lie superalgebras we consider and their root systems can be found in \cite{CP}.

Let $\mathfrak s$ be a Lie algebra or Lie superalgebra with a fixed Cartan subalgebra $\mathfrak h =\mathfrak h_{\bar{0}}$. A \emph{weight}  module $M$ is an $\mathfrak s$-module that is semisimple as $\mathfrak h$-module. The $\mathfrak h$-isotypic components of $M$ are the \emph{weight spaces} of $M$: we denote them by $M^{\lambda}$ for $\lambda \in \gh^*$.  The weight spaces of $M$ are superspaces. Every weight module $M$ has a well-defined support:
       $$ \supp  M = \{\mu \in \gh^* \; | \; M^{\mu} \neq 0 \}. $$

A \emph{weight}  module is \emph{bounded} if the dimension $(d_0|d_1)$ of any weight space of $M$ is less or equal to  $(a|b)$ for some fixed $a,b \in \mathbb Z_{\geq 0}$, i.e., $d_0\leq a$, $d_1\leq b$. The \emph{degree} $d (M)$ of a bounded weight module $M$ equals the maximum value of the sum $d_0+d_1$ over all weight spaces of $M$.

Each of our Lie superalgebras has (up to isomorphism) two natural modules which we denote by $V$ and $\Pi V$. These modules are weight modules, and for $\gg \gl (a|b)$ and $\sl (a|b)$ we assume that the weight spaces of  weight $\varepsilon_i$ in $V$ are purely odd and  the weight spaces of weight $\delta_j$ in $V$ are purely even. For $\osp (2a+1|2b)$ and $\osp (2a|2b)$   we make the opposite choice. We have
$$
\supp V = \begin{cases} \{ \varepsilon_i, \delta_j \; | \; i,j>0\} \mbox{ if } \gg = \sl(a|b)  \mbox{ or }  \gg = \gg \gl(a|b) ,\\ 
\{ 0, \pm \varepsilon_i, \pm \delta_j \; | \; i,j>0\} \mbox{ if } \gg = \osp(2a+1|2b),\\
\{ \pm \varepsilon_i, \pm \delta_j \; | \; i,j>0\} \mbox{ if } \gg = \osp(2a|2b).
\end{cases}
$$
For $\gg \gl (a|b)$ and  $\sl (a|b)$ modules $V_*$ and $\Pi V_*$ are also well defined. They are characterized by equalities  $\supp V_*= -\supp V$, $\supp \Pi V_*= -\supp \Pi V$, and by the fact that the weight spaces  of  weight $-\varepsilon_i$ in $V_*$ are purely odd and  the weight spaces of weight $-\varepsilon_i$ in $\Pi V_*$ are purely even. 

We now recall some facts about \emph{multiplicity free} weight $\mathfrak s$-modules for a finite-dimensional Lie algebra $\mathfrak s$, i.e., bounded weight $\mathfrak s$-modules $M$ with $d(M)=1$. Their classification has been part of a major effort to classify simple weight modules with finite-dimensional weight spaces. Some of the  main contributors have  been D. Britten, F. Lemire, S. Fernando, V. Futorny, G. Benkart,  O. Mathieu, and Mathieu's paper \cite{Mat} can be considered as the crown of this effort. It follows from a result of Fernando \cite{Fer} that for $\mathfrak s = \so (n)$,
$n\geq 5$ every multiplicity free simple weight $\so (n)$ is finite dimensional, hence is a trivial module, natural module, or a spinor module. For $\mathfrak s = \sp (2n)$ the only  multiplicity free simple finite-dimensional  $\mathfrak s$-modules are the trivial and the natural modules, and there is a ``coherent family'' of infinite-dimensional multiplicity free simple weight $\mathfrak s$-modules  \cite{BBL}, \cite{Mat}. For every Borel subalgebra $\mathfrak b \supset \mathfrak h$, there are precisely two nonisomorphic multiplicity free  simple $\mathfrak b$-highest weight modules in this family. These highest weight modules are known as oscillator or Shale-Weil modules, and every other infinite-dimensional multiplicity free simple weight module is obtained from one of them via twisted localization, see \cite{Mat}. For $\gs = \sl(n)$ the simple multiplicity free 
 weight modules have been classified in \cite{BBL} and have been further studied by O. Mathieu in \cite{Mat}. In this paper we will not refer to the description of all simple multiplicity free weight modules for $\sl(n)$ and $\sp(2n)$, but for understanding our results it is essential to know that simple multiplicity free weight modules, and more generally simple bounded weight modules, are well studied.

For $\gs = \sl(\infty), \sp (\infty), \so(\infty)$, simple bounded weight modules are described explicitly in \cite{GrP}. In the case of $\so(\infty)$, any bounded weight module is \emph{integrable}, i.e., it is a direct limit
of finite-dimensional  $\so(n)$--modules for $n \to \infty$. More precisely, if $M$ is a simple bounded  weight $\so(\infty)$-module, then $M$ is a trivial module, a natural module, or a direct limit of spinor modules. We refer to the latter direct limits simply as \emph{spinor  $\so(\infty)$-modules}. For $\gs = \sp(\infty)$ the result is similar. Namely, a simple bounded  weight  $\sp(\infty)$-module is a trivial module, a natural module, or a direct limit of simple multiplicity free infinite-dimensional  $\sp(2n)$--modules for $n \to \infty$. A difference with the case of $\so(\infty)$ is that a direct limit of simple multiplicity free infinite-dimensional modules is not integrable. We call such a direct limit a  
 \emph{simple weight $\sp(\infty)$-module of  oscillator type}. 

  In the sequel we will need the following general lemma about associative superalgebras.

 \begin{lemma}\label{associative} Let $A$ be an associative superalgebra and $X$ be a simple $A$-module. Then $X_{\bar 0}$ and $X_{\bar 1}$ are simple
   $A_{\bar 0}$-modules.
 \end{lemma}
 \begin{proof} If $Y\subset X_{\bar 0}$ (respectively, $Y\subset X_{\bar 1}$) is a proper nonzero $A_{\bar 0}$-submodule then $AY$ is an $A$-submodule of $Y$
   and 
   $(AY)_{\bar 0}=Y$ (respectively, $(AY)_{\bar 1}=Y$).
   \end{proof}

We conclude  Section \ref{prelim} with some facts concerning finite-dimensional Lie (super)algebras $\mathfrak s$. For a partition (equivalently, a Young diagram) $\mu$, let $\mathbb S_{\mu} \cdot$ denote the corresponding Schur functor.

\begin{proposition}\label{finrankintegrablesl} Let $\gs=\mathfrak{sl}(n)$ and $V$ be the defining $\gs$-module. If $n\geq |\mu|$ then $d(\mathbb S_\mu V)$ equals the dimension of the simple $S_{|\mu|}$-module $Z_{\mu}$ associated to $\mu$. 
\end{proposition}
\begin{proof} It suffices to consider the case $n=|\mu|$. Let $\{e_1,\dots,e_{n}\}$ be the standard $\gh$-eigenbasis of $V$.
  Let $\omega=\varepsilon_1+\dots+\varepsilon_n$. Then the weight space $(V^{\otimes n})^\omega$ has a structure of $W\times S_n$-module,
  where $W\simeq S_n$ is the Weyl group of $\mathfrak{sl}(n)$. Moreover, as an $S_n$-module $(V^{\otimes n})^\omega$ is isomorphic to the regular representation of $S_n$. Therefore, the isomorphism
  $$(V^{\otimes n})^\omega\simeq \bigoplus_\mu{} (\mathbb S_\mu V)^\omega\otimes Z_\mu$$
 forces $\dim (\mathbb S_\mu V)^\omega = \dim Z_\mu$. 
\end{proof}
\begin{lemma}\label{estimate} Let $\gs$ be a simple finite-dimensional Lie algebra, and $L(\mu), L(\nu)$ be simple finite-dimensional modules with respective  highest weights $\mu,\nu$. Then $d(L(\mu+\nu))\geq d(L(\mu))$.
\end{lemma}
\begin{proof} 
	Let $\pi:L(\mu)\otimes L(\nu)\to L(\mu+\nu)$ be the unique surjective homomorphism. Then the restriction of $\pi$ to $L(\mu)^\lambda\otimes L(\nu)^\nu$ is injective,  where $\lambda$ is a weight of $L(\mu)$ of maximal multiplicity. This implies the statement.
\end{proof}
\begin{lemma}\label{degreefinite} Let $\gs=\so(2n+1),\so(2n)$, or $\sp(2n)$. Then $L(\mu)$  is either multiplicity free or $d(L(\mu))\geq n-1$. 
\end{lemma}
\begin{proof} Let  $\omega_i$ be the $i$th fundamental weight of $\gs$. Set $\gs=\so(2n+1)$. Then $d(L(\omega_1))=d(L(\omega_n))=1$. For $k=2,\dots,n-1$ we have $L(\omega_k)\simeq\Lambda^k V$, thus $d(L(\omega_k))=\binom{n}{\lfloor{k/2}\rfloor}\geq n-1$.
  Next we note that
  $$d(L(2\omega_1))=d(S^2 V)=n,$$
  $$d(L(2\omega_n))=d(\Lambda^n V )=\binom{n}{\lfloor{n/2}\rfloor}\geq n-1,$$
  and
  $$d(L(\omega_1+\omega_n))\geq d(L(\omega_1)\otimes L(\omega_n))-d(L(\omega_n))=n.$$
  Consequently,  for $\mu = \omega_1,...,\omega_n, 2\omega_1, 2\omega_n, \omega_1 + \omega_n$ we see that $d(L(\mu))\geq n-1$. For any other $\mu$ the inequality follows from Lemma \ref{estimate}.

  The case of $\so(2n)$ is similar.

  Now let $\gs=\sp(2n)$. Then $d(L(\omega_1))=d(V)=1$. For $k>1$ we have $L(\omega_k)=\Lambda^k V/\Lambda^{k-2} V$. Hence $d(L(\omega_k))=\binom{n}{\lfloor{k/2}\rfloor}-\binom{n}{\lfloor{k/2}\rfloor-1}\geq n-1$.
  Next, $L(2\omega_1)$ is the adjoint representation and hence $d(L(2\omega_1))=n$. For $\mu \neq \omega_1,...,\omega_n, 2\omega_1$,  the statement follows again from Lemma \ref{estimate}.
  \end{proof}
  \begin{lemma}\label{osp12} Let $\gs=\osp(1|2n)$ and let $L(\mu)$ be a finite-dimensional simple $\gs$-module with highest weight $\mu$ relative to the Borel subsuperalgebra with simple roots $\delta_1 - \delta_2,...,\delta_{n-1} - \delta_n, \delta_n$. Assume
    $d(L(\mu))<n$. Then $\mu=\delta_1$, $\mu=0$, or $\mu=-\frac{1}{2}(\delta_1+\dots +\delta_n)$.
  \end{lemma}
  \begin{proof} The restriction of $L(\lambda)$ to $\gs_0=\mathfrak{sp}(2n)$ can have only simple constituents with highest weights $0$, $\delta_1$, or
    $-\frac{1}{2}(\delta_1+\dots+ \delta_n)$, $-\frac{1}{2}(\delta_1+\dots+\delta_{n-1})-\frac{3}{2}\delta_n$. The statement follows.
    \end{proof}
    
    In this paper a \emph{bounded} primitive ideal of $U(\gs)$ is defined as a primitive ideal which annihilates a simple bounded weight $\gs$-module. It is a result of \cite{PSer3} that if $M$ and $N$ are simple weight modules annihilated by the same bounded primitive ideal $I$, then $M$ and $N$ are bounded and $d(M) = d(N)$. This allows to define the \emph{degree} of a bounded primitive ideal $I \subset U(\gs)$ by setting $d(I) := d(M)$ for any simple bounded weight $\gs$-module $M$ annihilated by $I$.

    \begin{lemma}\label{finiterank} Let $\gs=\sp(2n)$, $\sl(n)$ and $I$ be a bounded primitive ideal of $U(\gs)$ of degree $d$. Assume that $U(\gs)/I$ is infinite dimensional. Then either $d\geq  \operatorname{rk}\gs - 1$ or $d=1$. If $d=1$ and $\gs=\sl(n)$, then $I=\Ann_{U(\gs)} L(a\omega_1)$ or
      $I=\Ann_{U(\gs)} L(a\omega_n)$ for some  $a\notin\mathbb Z_{\geq 0}$. If $d=1$ and $\gs=\sp(2n)$, then $I$ is  the Joseph ideal (annihilator of an oscillator module).  
\end{lemma}
\begin{proof} Assume first  $d>1$. The inequality $d\geq  \operatorname{rk}\gs - 1$   for $\gs=\sl(n)$  follows from Lemma 2.25 in \cite{GrP}. 

We proceed to show that $d \geq  \operatorname{rk}\gs  = n$ for  $\gs=\sp(2n)$. Theorem 12.2 in \cite{Mat} implies $d = \frac{1}{2^{n-1}}\dim \, L_{\so}(\lambda) $ for some simple finite-dimensional $\so (2n)$-module $L_{\so}(\lambda)$ of highest weight $\lambda = \sum_{i=1}\lambda_i \varepsilon_i$ with $\lambda_i \in 1/2 + \Z$. Since $d>1$, we have $\lambda \neq \omega_{n-1}, \omega_{n}$. Moreover, if $|\lambda_k | \neq | \lambda_{k+1} |$ for some $k\geq 1$, the stabilizer of $\lambda$ in the Weyl group has at most $k! (n-k)!$  elements. Therefore the orbit of $\lambda$ has at least $\binom{n}{k}2^{n-1}$ elements, implying $d \geq n$. Consider now the case when all absolute values $|\lambda_i |$ are equal. Under this assumption, there are  two possibilities: (i) all $\lambda_i$ are equal, or (ii) $\lambda_1=...= \lambda_{n-1} = - \lambda_n$. We set $\mu = \lambda - (\varepsilon_{n-1} +\varepsilon_{n})$ in case (i) and $\mu = \lambda - (\varepsilon_{n-1} -\varepsilon_{n})$ in case (ii). Then $\mu$ is a weight of $L_{\so}(\lambda)$ and  the  Weyl group orbit of $\mu$ has at least  $n 2^{n-1}$ elements. This implies again  $d \geq n$.
  \end{proof}

\section{Clifford and Weyl superalgebras and weight modules over them} \label{sec-clif-weyl}
\subsection{Definitions and main properties}
Let $a,b \in \Z_{\geq 0} \cup \{ \infty \}$. The \emph{Weyl superalgebra} $D(a|b)$ is the associative superalgebra with generators $\{x_i,\partial_i\mid i =1,...,a; -1,...,-b\}$ of  parity
$$\bar{x_i}=\bar{\partial}_i=\begin{cases} 0\,\,\text{if}\,\,i > 0\\ 1\,\,\text{if}\,\,i<0\end{cases},$$
satisfying the relations
$$[x_i,x_j]=[\partial_i,\partial_j]=0,\quad[\partial_i,x_j]=\delta_{ij},$$
where $[u,v]:=uv-(-1)^{\bar{u}\bar{v}}vu$ and  $\delta_{ij}$ is  Kronecker's delta.  The \emph{Clifford superalgebra} $Cl(a|b)$ is the associative superalgebra with generators $\{\xi_i,\eta_i\mid i =1,...,a; -1,...,-b \}$ with generators of parity
$$\bar{\xi}_i=\bar{\eta}_i=\begin{cases} 0\,\,\text{if}\,\,i > 0\\ 1\,\,\text{if}\,\,i<0\end{cases},$$
satisfying the relations
$$\{\xi_i,\xi_j\}=\{\eta_i,\eta_j\}=0,\quad\{\eta_i,\xi_j\}=\delta_{ij},$$
where $\{u,v\}:=uv+(-1)^{\bar{u}\bar{v}}vu$. In what follows, whenever $x_i, \partial_i, \xi_i,\eta_i$ are used we assume that the index $i$ is nonzero. 

We define a $\mathbb Z$-grading on $D(a|b)$ (respectively, on $Cl(a|b)$) by setting $\operatorname{deg} x_i:=1,\ \operatorname{deg} \partial_i:=-1$ (respectively, $\operatorname{deg} \xi_i:=1,\ \operatorname{deg} \eta_i:=-1$).
If $A=D(a|b)$ or $A=Cl(a|b)$ we denote by $A_{ev}$ the subsuperalgebra of elements of even degree, and by $A_n$ the subsuperspace of elements of degree $n$. Note that $A_{\bar{0}}$, $A_{0}$, and $A_{ev}$ are three different subsuperalgebras of $A$.

For $a$, $b\in\mathbb{Z}_{\geq 0}$, $D(a|b)$ (respectively, $Cl(a|b)$)  is naturally embedded in $D(a+1|b)$ and $D(a|b+1)$ (respectively, in $Cl(a+1|b)$ and $Cl(a|b+1)$), and 
$$D(\infty|\infty)=\varinjlim D(a|b),\quad Cl(\infty|\infty)=\varinjlim Cl(a|b).$$

\subsection{Connection to classical Lie superalgebras} \label{sebsec-conn}
Let $V_{2a|2b}$ be the subsuperspace of $D(a|b)$ with basis $\{x_i,\partial_i \; | \; -b \leq i \leq  a \}$. Then $V_{2a|2b}$ has an even anti-symmetric form given by the commutator map $[V_{2a|2b},V_{2a|2b}]\to\mathbb C$. The Lie superalgebra
$\osp(2b|2a)$ 
for which this form is invariant can be identified canonically with $S^2 V_{2a|2b} $. The symmetrization map 
$$
V_{2a|2b}^{\otimes 2} \to D(a|b), \; v\otimes w \mapsto \frac{1}{2} \left( v\otimes w  + (-1)^{\bar{v}\bar{w}} w \otimes v\right),
$$ 
factors through  $S^2 V_{2a|2b}$  and  defines a homomorphism of Lie superalgebras \\
$\osp(2b|2a) \to D(a|b).$ This   
induces a homomorphism of associative superalgebras $$\Phi_{a|b}:U(\osp(2b|2a))\to D(a|b).$$

Similarly, let $U_{2a|2b}$ be the subsuperspace of $Cl(a|b)$ with basis $\{\xi_i,\eta_i \; | \; -b \leq i \leq  a \}$. Then $U_{2a|2b}$ has an even  symmetric bilinear form given by the symmetrizer map $\{U_{2a|2b},U_{2a|2b}\}\to\mathbb C$. The Lie superalgebra
$\osp(2a|2b)$ 
for which this form is invariant can be identified canonically with $\Lambda^2 U_{2a|2b} $. The  alternization map 
$$
U_{2a|2b}^{\otimes 2} \to Cl(a|b), \; v\otimes w \mapsto \frac{1}{2} \left( v\otimes w  - (-1)^{\bar{v}\bar{w}} w \otimes v\right),
$$  
factors through  $\Lambda^2 U_{2a|2b} $  and  defines a homomorphism of Lie superalgebras\\ $\osp(2a|2b)\to Cl(a|b)$. This  
induces a homomorphism of associative superalgebras $$\Psi_{a|b}:U(\osp(2a|2b))\to Cl(a|b).$$

The Chevalley basis vectors $e_{\alpha}$ and the respective relations of the Lie superalgebras $\osp(2b|2a)$  (and also of $\osp(2a|2b)$) have been computed in \cite{FG},  \S 3.2. Up to scalar multiples, 
 the homomorphism $\Phi_{a|b}$ has the form 
 
\noindent $e_{\varepsilon_k  - \varepsilon_l } \mapsto x_{-l} \partial_{-k}, e_{-\varepsilon_k - \varepsilon_l }  \mapsto x_{-k} x_{-l}, e_{\varepsilon_k  + \varepsilon_l} \mapsto \partial_{-k} \partial_{-l} $, \\
$e_{-\delta_i - \delta_j} \mapsto x_{i}x_{j}$, $e_{-2\delta_i} \mapsto x_{i}^2$, $e_{\delta_{i} + \delta_{j}} \mapsto \partial_{i}\partial_{j}$, $e_{2\delta_i} \mapsto \partial_{i}^2$,\\
$e_{-\varepsilon_k + \delta_i} \mapsto x_{-k}\partial_{i}$, $e_{\varepsilon_k - \delta_i} \mapsto x_{i}\partial_{-k},  e_{-\varepsilon_k - \delta_i} \mapsto x_{-k} x_i, e_{\varepsilon_k + \delta_i} \mapsto \partial_{-k} \partial_{i},$ \\

\noindent and the homomorphism $\Psi_{a|b}$ has the form

\noindent $e_{\varepsilon_k  - \varepsilon_l } \mapsto \xi_{l} \eta_{k}, e_{-\varepsilon_k - \varepsilon_l }  \mapsto \xi_{k} \xi_{l}, e_{\varepsilon_k  + \varepsilon_l} \mapsto \eta_{k} \eta_{l}$, \\
$e_{-\delta_i - \delta_j} \mapsto \xi_{-i}\xi_{-j}$, $e_{-2\delta_i} \mapsto \xi_{-i}^2$, $e_{\delta_{i} + \delta_{j}} \mapsto \eta_{-i}\eta_{-j}$, $e_{2\delta_i} \mapsto \eta_{-i}^2$,\\
$e_{-\varepsilon_k + \delta_i} \mapsto \xi_{k}\eta_{-i}$, $e_{\varepsilon_k - \delta_i} \mapsto \eta_{k}\xi_{-i}, e_{-\varepsilon_k - \delta_i} \mapsto \xi_k\xi_{-i}, e_{\varepsilon_k + \delta_i} \mapsto \eta_k\eta_{-i},$ \\

\noindent  where $k\neq l$, $i \neq j$. 

\begin{lemma}\label{homomorphisms_Lie} The image of $\Phi_{a|b}$ coincides with $D(a|b)_{ev}$ and the image of $\Psi_{a|b}$ coincides with $Cl(a|b)_{ev}$.
  \end{lemma}
  \begin{proof} Let us consider $\Phi_{a|b}:U(\osp(2b|2a))\to D(a|b)$. For any $vw\in S^2 V_{2a|2b} $ we have $\Phi_{a|b}(vw)\in D(a|b)_2\oplus D(a|b)_0\oplus D(a|b)_{-2}$. Therefore
      $$\Phi_{a|b}(U(\osp(2b|2a)))\subset D(a|b)_{ev}.$$
      Moreover, the above formulas for  $\Phi_{a|b}$ show that $\Phi_{a|b}(\mathbb C\oplus \osp(2b|2a))$ is the span of $$S=\{1,x_i\partial_j,\partial_i\partial_j,x_ix_j\mid -b \leq i,j \leq a\}.$$ By a simple induction argument one  shows that $S$ generates $D(a|b)_{ev}$, and the statement follows. A similar argument applies  to $\Psi_{a|b}$. 
    \end{proof}

By $\C [x]$ we denote the symmetric superalgebra of the superspace with basis\\ $\{ x_i \; | \;  -b \leq i \leq a \} $.  The superspace  $\C [x]$   is  a simple faithful $D(a|b)$-module, and we call it the  \emph{defining} $D(a|b)$-module.  Furthermore, $\mathbb C[x]_{ev} = \mathbb C[x] \cap D(a|b)_{ev}$ is a simple faithful $D(a|b)_{ev}$-module and hence $\ker \Phi_{a|b}$ is a primitive ideal of $U(\osp(2b|2a))$. The pullback of $\mathbb{C}[x]_{ev}$ to
    $U(\osp(2b | 2a))$ is a simple highest weight module of $U(\osp(2b | 2a))$ of highest weight $\frac{1}{2}\left( \sum_{i=1}^b \varepsilon_i - \sum_{j=1}^a \delta_j \right)$ relative to the Borel subsuperalgebra with positive roots 
    $$\delta_p\pm\delta_q\,\text{for}\,p>q,\,2\delta_{p},\,\delta_p\pm\varepsilon_q,\,\varepsilon_p\pm\varepsilon_q\,\text{for}\,p<q,$$
where the sum $\sum_{i=1}^b \varepsilon_i-\sum_{j=1}^a \delta_j$ is an infinite formal sum if $b=\infty$ or $a=\infty$.

    Similarly, the \emph{defining} $Cl(a|b)$-module $\Lambda[\xi]$ is the exterior superalgebra of the superspace with basis $\{ \xi_i \, | \, -b\leq i \leq a \}$. The module $\Lambda[\xi]$ is a  simple and faithful  $Cl(a|b)$-module. Furthermore, $\Lambda[\xi]_{ev} = \Lambda[\xi] \cap Cl(a|b)_{ev}$ is a simple faithful $Cl(a|b)_{ev}$-module and hence 
     $\ker \Psi_{a|b}$ is a primitive ideal of $U(\osp(2a|2b))$. The pullback of $\Lambda [\xi]_{ev}$ is a simple highest weight $\osp(2a|2b)$-module with highest weight $\frac{1}{2}\left( \sum_{i=1}^b \varepsilon_i - \sum_{j=1}^a \delta_j \right)$, and it is isomorphic  to the pullback of $\mathbb{C}[x]_{ev}$.    These two  isomorphic highest weight modules have purely even highest weight spaces. Next, the pullback of the odd-degree part $\Lambda [\xi]_{odd}$ of  $\Lambda [\xi]$ is a simple $\osp(2a|2b)$-module  with highest weight $\frac{1}{2}\left( \sum_{i=1}^a \varepsilon_i - \sum_{j=1}^b \delta_j \right) - \delta_1$. The pullbacks of $\Lambda [\xi]_{odd}$ and $\mathbb{C}[x]_{odd}$ are isomorphic and have purely odd  highest weight  spaces.

      The  pullbacks of  $\mathbb{C}[x]_{ev}$ and $\mathbb{C}[x]_{odd}$ (equivalently,  of 
     $\Lambda [\xi]_{ev}$ and  $\Lambda [\xi]_{odd}$),
         together with their counterparts with changed parity, are four pairwise nonisimorphic $\osp(2a|2b)$-modules, which we define to be \emph{spinor-oscillator representations}. A general spinor-oscilator representation is the twist of some of these four modules by an automorphism of the Lie superalgebra $\osp(2a|2b)$. For $b=0$ (respectively, for $a=0$) the spinor-oscillator representations are nothing but the spinor representations of $\so(2a)$ (respectively, the oscillator or Shale-Weil representations of $\sp(2b)$). (It is well known that for a fixed Borel subalgebra there are precisely two isomorphism classes of purely even spinor or, respectively, oscillator representations.)

The isomorphisms of the pullbacks of $\C [x]_{ev}$ and  $\Lambda [\xi]_{ev}$ imply the following.
    \begin{corollary}\label{isomorphism_even} $\ker \, \Phi_{b|a}=\ker\, \Psi_{a|b}$ and hence $Cl(a|b)_{ev}$ and $D(b|a)_{ev}$ are isomorphic associative superalgebras.
    \end{corollary}
  
    \begin{remark} It is known that $Cl(a|b)$ is the universal enveloping algebra of the Jordan superalgebra $U_{2a|2b}\oplus\mathbb C 1$, while $D(a|b)$ is the quotient of the universal enveloping algebra of the Heisenberg
      superalgebra $V_{2a|2b}\oplus\mathbb C z$ by the ideal $(z-1)$.  Furthermore, it is easy to see that the superalgebras $D(b|a)$ and $Cl(a|b)$ are not isomorphic unless $ab=0$.
    \hfill${\bigcirc}$    \end{remark}

    Now, we note that $\Psi_{a|b}(\osp(2a|2b))\oplus V_{2a|2b}$ is closed under supercommutator, and the corresponding Lie superalgebra is isomorphic to $\osp(2a+1|2b)$. Hence we have a surjective homomorphism
    $$\Theta_{a|b}: U(\osp(2a+1|2b))\to Cl(a|b).$$
     The explicit formulas for  $\Theta_{a|b}$ are the same as those for $\Psi_{a|b}$, with the following addition:
    $$e_{\varepsilon_k} \mapsto \eta_{k}, \, e_{-\varepsilon_k }  \mapsto \xi_{k}, \, e_{-\delta_i } \mapsto \xi_{-i}, \, e_{\delta_{i}} \mapsto \eta_{-i}.$$

    The pullback via $\Theta_{a|b}$ of the defining $Cl(a|b)$-module  $\Lambda[\xi]$  is an irreducible $\osp(2a+1|2b)$-module with highest weight  $\frac{1}{2}\left( \sum_{i=1}^a \varepsilon_i - \sum_{j=1}^b \delta_j \right)$ with respect to the Borel subsuperalgebra with positive roots 
$$\delta_p\pm\delta_q\,\text{for}\,p>q,\,\delta_p,\,2\delta_{p},\,\delta_p\pm\varepsilon_q,\,\varepsilon_p\pm\varepsilon_q\,\text{for}\,p<q,\,\varepsilon_p.$$

  We call this highest weight module, together with its counterpart with changed parity, a \emph{spinor-oscillator representation} of $\osp(2a+1|2b)$. Moreover, $\ker\Theta_{a|b}$ is the primitive ideal of a spinor-oscillator representation of $\osp(2a+1|2b)$.

      We note also that $\ggl(a|b)$ is the reductive part of a parabolic subalgebra of $\osp(2a|2b)$, and by composing {the injection} $\ggl(a|b)\hookrightarrow \osp(2a|2b)$ with $\Phi_{b|a}$ we  obtain a surjective homomorphism
      \begin{equation}\label{eq1} U(\ggl(a|b))\to D(b|a)_0\simeq Cl(a|b)_{0}.
        \end{equation}
        Similarly, the embedding $\ggl(a|b)\hookrightarrow\osp(2b|2a)$ induces a surjective homomorphism
\begin{equation}\label{eq2}
  U(\ggl(a|b))\to D(a|b)_0\simeq Cl(b|a)_{0}.
\end{equation}
We denote by $\Upsilon^-_{a|b}$ the restriction of the homomorphism (\ref{eq1}) to $U(\mathfrak{sl}(a|b))$, and  $\Upsilon^+_{a|b}$ the
restriction of the homomorphism (\ref{eq2}) to $U(\mathfrak{sl}(a|b))$.
      
      We will use the homomorphisms $\Upsilon^\pm_{a|b}$  in Section \ref{sec-sl-bounded}.

           \subsection{Tensor product isomorphisms}      
           Let $Cl^\dagger(a|b)$ (respectively, $D^\dagger(a|b)$) be the superalgebra defined by the same generators and relations as $Cl(a|b)$  (respectively, $D(a|b)$), but where the generators $x_i, \partial_i$ (respectively, $\xi_i, \eta_i$) for $i>0$ are endowed with  the opposite parity. 
                      
           Then one can check  that the correspondence $\xi_{-i} \mapsto x_i, \eta_{-i} \mapsto \partial_i$, $i=1,...,b$, defines an isomorphism of superalgebras
           \begin{equation}\label{D-iso}
             Cl(0|b)\simeq D^\dagger(b|0),
           \end{equation}
           and the correspondence $\xi_{i} \mapsto x_{-i}, \eta_{i} \mapsto \partial_{-i}$, $i=1,...,b$, defines an isomorphism of superalgebras
       \begin{equation}\label{Cl-iso}    
         Cl^\dagger(b|0)\simeq D(0|b).
         \end{equation}
      
      \begin{lemma}\label{isomorphism_tensor} We have the following isomorphisms of associative superalgebras
                   \begin{equation}\label{D-dagger}
                   D(a|b)\simeq D(a|0)\otimes D(0|b)\simeq D(a|0)\otimes Cl^\dagger(b|0),
                   \end{equation}
      \begin{equation}\label{Cl-dagger} Cl(a|b)^\dagger\simeq D(0|a)\otimes D^\dagger(b|0).
          \end{equation}
    
      \end{lemma}
      \begin{proof}  The isomorphisms \eqref{D-dagger} follow from (\ref{Cl-iso}) and from the fact that $x_i,\partial_i$ commute with $x_{-j},\partial_{-j}$ for all positive $i,j$. Similarly, the  isomorphism \eqref{Cl-dagger} follows from  (\ref{D-iso}) and from the fact that $\xi_i,\eta_i$ anticommute with $\xi_{-j},\eta_{-j}$ for all positive $i,j$. 
   \end{proof}
   
   \begin{corollary}\label{even-iso} We have isomorphisms of (purely even) associative algebras:
   \begin{itemize}
\item[(a)]
$D(a|b)_{\bar 0}\simeq D(a|0)\otimes Cl(b|0)_{ev},\quad Cl(a|b)_{\bar 0}\simeq Cl(a|0)\otimes D(b|0)_{ev};$
\item[(b)]   $(D(a|b)_{ev})_{\bar 0} \simeq D(a|0)_{ev}\otimes D(0|b)_{ev},\quad (Cl(a|b)_{ev})_{\bar 0}\simeq Cl(a|0)_{ev}\otimes Cl(0|b)_{ev}.$
\end{itemize}
     \end{corollary}
     \begin{proof} Part (a)  is a consequence of the existence of  isomorphisms $Cl^{\dagger}(b|0)_{\bar 0}\simeq Cl(b|0)_{ev}$ and $ D^\dagger(b|0)_{\bar 0}\simeq D(b|0)_{ev}$. Part (b) follows straightforwardly from part (a).
       \end{proof} 

      \subsection{Simple weight modules over Clifford and Weyl algebras} \label{subsec-simple} In the rest of the paper, $A$ stands for $D(a|b)$ or $Cl(a|b)$ unless a restriction on $A$ is made explicit. Set $u_i:=x_i\partial_i$ ($ i  \neq 0$) for $A=D(a|b)$, $u_i:=\xi_i\eta_i$ ($ i  \neq 0$)  for $A=Cl(a|b)$, and define $$\gh_A:=\operatorname{span}\{u_i \mid i \neq 0\}.$$

Let $\{\zeta_i \; | \; i \neq 0\}\subset\gh_A^*$ be the system dual to $\{u_i \; | \; i \neq 0 \}$. 
                Then $\gh_A^* = \prod_{i \neq 0} \C \zeta_i$. For convenience, we will write the elements of $\gh_A^*$ as formal (possibly infinite) sums  $\sum_{i \neq 0}a_i \zeta_i $.
  We set  $$Q_A:=\bigoplus_{i \neq 0} \mathbb Z \zeta_i.$$ 
 
  One can easily see that the abelian Lie algebra $\gh_A$ acts semisimply on $A$ via the adjoint action.  In other words, 
     $$A= \bigoplus_{\alpha\in R_A \sqcup \{ 0 \} } A^\alpha,\quad A^\alpha=\{x\in A\mid \ad_h(x)=\alpha(h)x \mbox{ for every } h \in \gh_A\},$$
      and $R_A$ is the set of all $\alpha \in Q_A \setminus \{ 0 \}$ such that $A^{\alpha} \neq 0$.          
       If $A =Cl(a|b)$, then   $$R_A \sqcup \{ 0\}= \left\{ \sum_{i \neq 0} a_i \zeta_i \in Q_A \; |\; a_i=0 \mbox{ for almost all $i$, and }a_i\in \{0,1 \} \mbox{ for } i > 0   \right\}.$$
       If $A =D(a|b)$,   then $$R_A \sqcup \{ 0\}= \left\{ \sum_{i\neq 0} a_i \zeta_i \in Q_A \; |\; a_i=0 \mbox{ for almost all $i$, and }a_i\in \{0,1 \} \mbox{ for } i < 0   \right\}.$$

     Moreover, for  $A = Cl(a|b)$ we have $\xi_i \in A^{\zeta_i}$, $\eta_i \in A^{-\zeta_i}$ if $i \neq 0$. For $A = D(a|b)$ we have  $x_i \in A^{\zeta_{i}}$, $\partial_i \in A^{-\zeta_{i}}$ if $i\neq 0$.

   Note that each superspace $A^\alpha$  is purely even or purely odd.  Define the \emph{parity function on $Q_A$} to be the homomorphism of abelian groups $p:Q_A\to\mathbb Z_2 $ which records the parity of the superspace $A^\alpha$ for $\alpha \in R_{A}$. 
Explicitly, 
      $p(\zeta_i)=0$ for $i>0$ and $p(\zeta_j)=1$ for  $j<0$.

      \begin{lemma}\label{rootsA} \begin{enumerate}
        \item[(a)]  The subalgebra $H_A:=A^0$ is generated by $\gh_A$.
        \item[(b)] If $A=D(a|b)$ then $H_A$ is isomorphic to $\mathbb C[u]/(u_i^2-u_i)_{i<0}$.
        \item[(c)]  If $A=Cl(a|b)$ then $H_A$ is isomorphic to $\mathbb C[u]/(u_i^2-u_i)_{i > 0}$.
          \item[(d)] Every  root space $0\neq A^{\alpha}$ is a cyclic $H_A$-module.
          \end{enumerate}
        \end{lemma}
        \begin{proof} Straightforward computations.
        \end{proof}

Set
$$
\gh^{\vee}_A:= \left\{\mu\in\gh_A^*\mid\mu(u_i)=0,1 \mbox{ where } i<0   \mbox{ for  }A=D(a|b)  \mbox{ and } i>0  \mbox{ for }A=Cl(a|b)\right\}.
      $$

In what follows, we refer to the elements of $\gh^{\vee}_A$ as to the \emph{weights of $A$}. An element $\mu$ of $\gh^{\vee}_A$ is a formal sum
$$\mu =  \sum_{i\neq 0}\mu_i \zeta_i $$
 with the only restriction that $\mu_i\in \{0,1 \}$ for  $i>0$ if $A = Cl(a|b)$,  and  $\mu_i\in \{0,1 \}$ for $i<0$ if $A = D(a|b)$.  Note that  $\gh^{\vee}_A$ is not a vector space.

\begin{remark}\label{weightcorrespondence} Let $\gg$ be a Lie superalgebra isomorphic to $\mathfrak{osp}(2a|2b)$ (respectively,
       $\mathfrak{osp}(2a+1|2b)$) with fixed Cartan subalgebra $\gh$. Set $A=Cl(a|b)$ and let $F:U(\gg)\to A$ be the homomorphism $\Psi_{a|b}$ (respectively, $\Theta_{a|b}$).  Then $F(U(\gh))=H_A$. We have
       $$\operatorname{Specm}H_A=\gh^*_A,\quad \operatorname{Specm}U(\gh)=\gh^*, $$
       where $\operatorname{Specm}$ denotes  maximal spectrum. Set
       $$\tau:=\frac{1}{2}(\sum_{i>0}\varepsilon_i-\sum_{j>0}\delta_{j}).$$ The map $f:\gh_A^*\to \gh^*$ induced by $F$ is not linear but affine, i.e.,
$$f(\mu+\nu)=f(\mu)+f(\nu)-f(0)$$
       with $f(0)=\tau$. Moreover,
       $$f(\zeta_i)=\begin{cases} \varepsilon_i-\tau,\ i>0 \\ \delta_{-i}-\tau,\ i<0 \end{cases}.$$

       Similarly if $\gg=\mathfrak{osp}(2b|2a)$, $A=D(a|b)$ and $F:=\Phi_{a|b}$, we have
               $$f(\zeta_i)=\begin{cases} \varepsilon_{-i}-\tau,\ i<0 \\ \delta_{i}-\tau,\ i>0 \end{cases}.$$
       \hfill${\bigcirc}$
     \end{remark}

Let $\CC_\mu$ be the unique $(1|0)$-dimensional $H_A$-module on which $\gh_A$ acts via $\mu$.  According to Lemma \ref{rootsA}(a)-(c) every simple $H_A$-module is one-dimensional and is isomorphic to $\CC_{\mu}$ for some $\mu\in \gh^{\vee}_A$. 
        
                An $A$-module $X$ is a \emph{weight module} if $X$ is semisimple as an $H_A$-module, i.e., if $X$ has a  decomposition
        $$X=\bigoplus_{\mu\in \gh^{\vee}_A} X^\mu,$$
        where $X^{\mu} := \{ x \in X \; | \; hx=\mu(h)x \mbox{ for every }h \in \mathfrak h_A\}$ is the \emph{$\mu$-weight space} of $X$. The \emph{support of} a weight module $X$ is  
        $$ \supp  X = \{\mu \in \gh^{\vee}_A \; | \; X^{\mu} \neq 0 \}. $$

\begin{lemma}\label{pure}
Let $X$ be a simple weight $A$-module. Then the weight spaces of $X$ are purely even or purely odd. Hence $X$ and $\Pi X$ are never isomorphic.
\end{lemma}
\begin{proof}   Let $0 \neq x\in (X^\mu)_{{\kappa}}$, where ${\kappa} \in \mathbb Z_2 $. Then
  $$X=Ax=\bigoplus_{\alpha\in Q_A} (A^\alpha x = X^{\alpha + \mu}),$$
  i.e., all nonzero vectors in $X^{\alpha + \mu}$ are purely even (respectively, purely odd) if ${\kappa} + p(\alpha) = \bar{0}$ (respectively, if ${\kappa} + p(\alpha) = \bar{1}$).
    \end{proof}

  For the remainder of the paper we fix an extension  of the parity function $p:Q_A\to\Z_2$ to a  map $p:\hh_A^{\vee}\to \ZZ_2$ satisfying $p(\mu+\alpha)=p(\mu)+p(\alpha)$ for any $\alpha\in Q_A$
          and any $\mu\in\hh_A^{\vee}$. Note that such an extension is not unique.
          
 We call a weight $A$-module $X$ \emph{preferred} if for any $\mu \in \supp X$, the weight space $X^\mu$ is purely even if $p(\mu)=\bar 0$ and the weight space $X^\mu$ is purely odd if $p(\mu)=\bar 1$.  Lemma \ref{pure} implies that, if $X$ is a simple weight module then exactly one of the modules $X$ or $\Pi X$ is preferred. 
Moreover, any weight $A$-module
      $X$ decomposes uniquely into a direct sum $X_1\oplus\Pi X_2$ for some preferred modules $X_1$ and $X_2$.

        \begin{proposition}\label{dagger} The category of preferred weight $A^\dagger$-modules is equivalent to the category of preferred weight $A$-modules as an abelian category.
        \end{proposition}
        \begin{proof}  The superalgebra $A^\dagger$ has its own parity function $p^\dagger:Q_{A^\dagger}\to\Z_2$ with the property
          $p^{\dagger} (\alpha) = 1$ for $\alpha = \varepsilon_i,\delta_j$. We can extend this function 
          to  a  map $p^\dagger:\hh_{A^\dagger}^{\vee}\to \ZZ_2$ satisfying $p^\dagger(\mu+\alpha)=p^\dagger(\mu)+p^\dagger(\alpha)$ for any
          $\alpha\in Q_{A^\dagger}$. Then, for a preferred weight module $X$ we set
          $$X^\dagger:=\bigoplus_{\mu\in\supp X} \Pi^{p^\dagger(\mu)-p(\mu)}X^\mu.$$
          It is clear that ${\cdot}^\dagger$ is a functor from the category of preferred weight $A$-modules to the category of preferred weight $A^\dagger$-modules. Moreover, the functor $({\cdot}^{\dagger})^\dagger$ is isomorphic to the identity functor.
        \end{proof}

In order to proceed with our study of  weight $A$-modules, for any $\mu \in \gh_A^{\vee}$ we  introduce a certain multiplicity free weight $A$-module $F(\mu)$ such that $\mu\in \operatorname{supp}F(\mu)$.

First, assume $A =D(a|b)$ and fix $\mu\in  \gh^{\vee}_A$. We can write $\mu=\{\mu_i\}$ with $\mu_i\in\mathbb C$ for $i>0$ and $\mu_i=0,1$ for $i<0$. Let $B$ be the subalgebra in $D(0|b)$ generated by all $x_i$ for  $i<0$ such that $\mu_i=1$,
and by all $\partial_i$ for  $i<0$ such that $\mu_i=0$. Then  $B$ is a local supercommutative  algebra, and we denote by $J$ its maximal ideal. 

Set
$R:=\mathbb C[x_i,x_i^{-1}]_{i>0}$. Consider the $D(a|0)$-module 
$F^+(\mu):=Rx^\mu$ defined by the relations $\partial_ix^\mu=\mu_ix_i^{-1}x^{\mu}$ and the $D(0|b)$-module $F^-(\mu):=D(0|b)\otimes_B(B/J)$. Finally using the first isomorphism of (\ref{D-dagger}), we define the $A$-module $F(\mu)$ by setting
$F(\mu):=F^+(\mu)\otimes \Pi^{p(\mu)} F^-(\mu)$.

Now let $A=Cl(a|b)$. Here we use the isomorphism (\ref{Cl-dagger}), and set
$$F(\mu):= \Pi^{p(\mu)} (F^-(\mu)\otimes F^+(\mu)^\dagger)^\dagger,$$
where now $F^-(\mu)$ is a $D(0|a)$-module  and  $F^+(\mu)$ is a $D(b|0)$-module.

By construction, $\mu\in\supp F(\mu)$ and all weight spaces of $F(\mu)$ are $1$-dimensional.

\begin{lemma}\label{indecomposable} The $A$-module $F(\mu)$  is preferred, indecomposable, and has a simple socle (i.e., a simple submodule which is contained in any
  nonzero submodule of $F(\mu)$). Under the assumption  $\mu_i\notin \mathbb Z$ for all $i>0$ if $A=D(a|b)$, and $\mu_j\notin \mathbb Z$ for all $j<0$ if $A=Cl(a|b)$, the module $F(\mu)$ is simple.
\end{lemma}
\begin{proof} 

Let $A=D(a|b)$. The fact that $F(\mu)$  is preferred follows directly from the definition of  $F(\mu)$.

Define the weight $\tilde{\mu} \in\supp F(\mu)$ by setting
  $$\tilde{\mu}_i:=\begin{cases}\mu_i\ \text{if}\ i<0\ \text {or}\ \mu_i\not\in\mathbb Z\\ 0\ \text{otherwise} \end{cases}.$$
  We claim that $F(\mu)^{\tilde{\mu}}$ generates a  simple submodule of $F(\mu)$ which is the socle of  $F(\mu)$. Indeed, note that if $\nu \in\supp F(\mu)$, the construction of $F(\mu)$
  shows that  the map $F(\mu)^\nu\to F(\mu)^{\nu+\varepsilon_i}$  of multiplication by $x_i$ is an isomorphism for all positive $i$, and that
  the map $F(\mu)^\nu\to F(\mu)^{\nu-\varepsilon_i}$ of application of  $\partial_i$
  is an isomorphism iff
  $\nu_i\neq 0$. Furthermore, for $i<0$ the map $F(\mu)^\nu\to F(\mu)^{\nu+\delta_i}$ of multiplication by $x_i$ is an isomorphism iff
  $\nu+\delta_i\in \supp F(\mu)$, and similarly  the map $F(\mu)^\nu\to F(\mu)^{\nu-\delta_i}$ of application of  $\partial_i$  is an isomorphism
  iff $\nu-\delta_i\in \supp F(\mu)$. Consequently, the cyclic submodule of $F(\mu)$ generated  by any nonzero weight vector contains the weight space $F(\mu)^{\tilde{\mu}}$.  This proves our claim, and we see that $F(\mu)$ is  indecomposable as it has a simple socle.

  Finally, if $\mu_i\notin \mathbb Z$ for all $i>0$ then  $\mu=\tilde\mu$ and  $F(\mu)$ is simple. 
  
  The case of  $A=Cl(a|b)$ is handled in a similar manner. 
  \end{proof}
  
          For $\mu, \nu \in \gh^{\vee}_A$ we write $\mu \approx \nu$ if $\mu - \nu \in Q_A$ and the respective sets of indices $i$ for which $\mu_i \in \mathbb Z_{\geq 0}$ and $\nu_i \in \mathbb Z_{\geq 0}$ coincide. 

        \begin{theorem}\label{thm:weightA}

          (a) Every simple weight $A$-module is multiplicity free.

          (b)  For every $\mu\in \gh^{\vee}_A$, up to isomorphism,  there exist precisely two simple $A$-modules $X(\mu)$ and $\Pi X(\mu)$ whose supports contain  $\mu$, and such that $X(\mu)$ is preferred.
          
            (c)  $ \supp X(\mu) = \left\{\lambda \in \gh_A^\vee  \; | \; \lambda \approx \mu \right\}$.
                        
          (d) Let $\mu - \nu \in Q_A$.      The modules $X(\mu)$ and $X(\nu)$ are isomorphic if and only if  $\mu  \approx \nu$.
        \end{theorem}
        \begin{proof} Set $P(\mu):=A\otimes_{H_A}\left( \Pi^{p(\mu)}\CC_\mu \right)$ for $\mu\in \gh^{\vee}_A$.  Then by Frobenius reciprocity
          $\operatorname{Hom}_A(P(\mu),F(\mu))\neq 0$. Hence  the weight space $P(\mu)^{\mu}$ is nonzero  and generates $P(\mu)$. Since each each weight space of $P(\mu)$  is a cyclic $H_A$-module (Lemma \ref{rootsA}(d)), the $A$-module $P(\mu)$ is multiplicity free.

          Therefore the sum $N$ of all submodules $Z$ of $P(\mu)$ with $Z^{\mu} = 0$ constitutes the 
 unique maximal proper submodule of $P(\mu)$.  Since $P(\mu)$ is multiplicity free, 
         the quotient
          $X(\mu):=P(\mu)/N$ and the module $\Pi X(\mu)$ are (up to isomorphism) the only two simple $A$-modules  whose supports contain  $\mu$. Note that  $X(\mu)$ is preferred, while  $\Pi X(\mu)$ is not. This proves (a) and (b).
          
          (c). It follows from (b) that the supports of non-isomorphic  simple modules are disjoint. It remains to check that
          $\supp X(\mu)$ is exactly the equivalence class of $\mu$. Let $A=D(a|b)$ and $v\in X(\mu)^{\nu}$. Then $x_iv\neq 0$
          (respectively, $\partial_i v\neq 0$) iff $\nu+\varepsilon_i\approx \nu$ (respectively, $\nu+\varepsilon_i\approx \nu$). Therefore,
          $X(\mu)^\nu\neq 0$ iff $\nu\approx\mu$. The case $A=D(a|b)$ is similar.

            (d).           Direct corollary of (c).
          \end{proof}

Next we would like to decompose the simple weight $A$-modules in accordance with the isomorphisms \eqref{D-dagger} and \eqref{Cl-dagger}. We start by discussing  weight modules of $A = Cl(b|0)$ and $A = D(0|b)$. In these cases we identify the  subsets $\mathbb A$ of $\mathbb Z \cap [1,b]$ (where $b=\infty$ is possible)  with the weights of $A$ via the map  $$\mathbb A \mapsto  \zeta_{\mathbb A},$$ where $\zeta_{\mathbb A} = \sum_{i \in \mathbb A} \zeta_i$ for $A = Cl(b|0)$ and $\zeta_{\mathbb A} = \sum_{i \in \mathbb A} \zeta_{-i}$ for $A = D(0|b)$. Accordingly, we  write $X(\mathbb A)$ instead of $X(\zeta_{\mathbb A})$. 

        \begin{lemma} \label{Clifford}   Let $A (b) = Cl(b|0)$ or $ A (b) = D(0|b)$. 
          
          (a) If $b < \infty$, then the category of preferred weight $ A (b)$-modules is semisimple and has,  up to isomorphism,  one simple object $X({\emptyset})$.

          (b)   If $b  = \infty$ then, up to isomorphism,  the simple preferred weight $ A (b)$-modules can be enumerated by equivalence
          classes of subsets of $\mathbb{Z}_{> 0}$ with respect  to the following equivalence relation: $\mathbb A$ is equivalent to $\mathbb B$ if the symmetric
          difference $\mathbb A\triangle \mathbb B$ is finite. In other words, up to isomorphism, there is exactly one simple weight $A(b)$-module $X ({\mathbb A})$ corresponding to ${\mathbb A}$. 
          
         (c)  We have 
           $X({\mathbb A}) \simeq X({\mathbb B)}$ if and only if $\mathbb A\triangle \mathbb B$ is finite.
          
        \end{lemma}
        \begin{proof} Claim (a) for  $ A (b) = Cl(b|0)$ is an immediate consequence of the fact that $ A (b)$  is a matrix algebra. The case  $ A (b) = D(0|b)$ with  $b < \infty$ follows from Proposition \ref{dagger}.

         Claim (b)  follows from Theorem \ref{thm:weightA}(b).

For part (c), we note that 
$\zeta_{\mathbb A} \approx \zeta_{\mathbb B} $ if and only if $\mathbb A  \Delta \mathbb B $ is finite. 
                  \end{proof}

          \begin{proposition}\label{description-of-simple}

   (a)         Every simple preferred weight $D(a|b)$-module $X$ is isomorphic to $X^+\otimes (X^-)^\dagger$  for some simple preferred weight $D(a|0)$-module $X^+$ and some simple preferred weight $Cl(b|0)$-module $X^-$.

   (b) Every simple preferred weight $Cl(a|b)$-module $X$  is isomorphic to $(\left(X^+)^\dagger\otimes (X^-)^\dagger \right)^{\dagger}$  for some simple preferred weight $Cl(a|0)$-module $X^+$ and some
   simple preferred weight $D(b|0)$-module $X^-$.
   
 \end{proposition}
 \begin{proof} We prove (a) since (b) is similar. For any weight  $\mu\in\supp X$ we can choose  a simple preferred weight $D(a|0)$-module $X^+$  and a simple preferred weight $Cl(b|0)$-module $X^-$ so that $\mu\in \supp \left( X^+\otimes (X^-)^\dagger \right)$.  Moreover, it is clear from the construction that the module $X^+\otimes (X^-)^\dagger$ is simple.
 Therefore  Theorem \ref{thm:weightA} implies the claim. {\color{red}} 
    \end{proof}

    \subsection{Categories of weight modules over Clifford and  Weyl   algebras}

    Let $\cW_{A}$ denote the category of preferred weight $A$-modules. To study the category  of all weight $A$-modules, it suffices to study
    the category  $\cW_{A}$. Indeed,  since every weight $A$-module decomposes canonically as $X_1 \oplus \Pi X_2$ where $X_1$ and $X_2$ are preferred,  the morphisms in the category of all weight $A$-modules are recovered by the morphisms in the category $\cW_{A}$ (the latter morphisms necessarily preserve the $\mathbb Z_2$-grading).

  Recall the $A$-module $P(\mu)$ introduced in  the proof of Theorem \ref{thm:weightA}.

  \begin{lemma}\label{pims} The $A$-module $P(\mu)$ is an indecomposable projective object in the category $\mathcal W_A$.
   The category  $\mathcal W_A$ has enough projectives.    
  \end{lemma}
  \begin{proof} By Frobenius reciprocity we have
    $$\Hom_A(P(\mu),X)\simeq\Hom_{\gh_A}(\Pi^{p(\mu)}\CC_\mu,X)\simeq X^\mu$$
for any preferred module $X$ in $\mathcal W_A$.     This implies the projectivity of $P(\mu)$. The indecomposability of  $P(\mu)$ follows from the fact that $P(\mu)$ has a unique maximal proper submodule.

    Noting that any $X\in \cW_{A}$ is a quotient of $\bigoplus_{\mu\in\supp X}P(\mu)\otimes X^\mu$, we see that  $\mathcal W_A$ has enough projectives.

    \end{proof}
      
      We introduce the following equivalence relation on the set of weights $\gh^{\vee}_A$:
 $\mu\sim\nu \Longleftrightarrow \mu\in\nu+Q_A$. Note that  the relation $\sim$ is weaker than the relation $\approx$, i.e.,   $\mu \approx\nu$ implies $\mu\sim\nu$.  Let $\Gamma$ denote a $\sim$-equivalence class in $\gh^{\vee}_A$, and let $\cW_A^\Gamma$ be the full subcategory of $\cW_A$ with objects $X$ satisfying $\supp X\subset\Gamma$.
 Since the support of every indecomposable weight $A$-module $X$ belongs to $\Gamma$ for some class $\Gamma$, we have a decomposition
 $$\cW_A=\prod \cW_A^\Gamma.$$

        \begin{proposition}\label{blocks} The subcategories $\cW_A^\Gamma$ are blocks of $\cW_A$. 
 \end{proposition}
 \begin{proof} If $X$ and $X'$ are two simple weight $A$-modules from $\cW_A$ satisfying  $\mu\sim\nu$ for some 
  $\mu\in \supp X$ and
   $\nu\in\supp X'$, then the modules $X$ and $X'$  occur as simple constituents in the $A$-module $F(\mu)$.  We know from Lemma \ref{indecomposable} that  $F(\mu)$ is a preferred indecomposable module. This implies the assertion.
   \end{proof}

   \begin{lemma}\label{semisimplicity} If $A=Cl(a|0)$ or $A=D(0|b)$ then $\cW_A$ is a semisimple category.
   \end{lemma}
   \begin{proof} It suffices to prove that every indecomposable projective module $P \in \cW_A$ is simple. For this, note that $P$ is an object of $\cW_A^\Gamma$ for some $\Gamma$, and let $\mu  \in \gh_A^{\vee}$ belong to $\supp P$. Then $\Hom_{H_A} (\Pi^{p(\mu)}\C_{\mu}, P) \neq 0$ and Frobenius reciprocity yields a nonzero homomorphism $P \to P(\mu) = A\otimes_{H_A}\Pi^{p(\mu)} \mathbb C_\mu$.  The key observation is that under the assumption  $A=Cl(a|0)$ or $A=D(0|b)$,  the $A$-module $P(\mu)$ is simple. This together with the projectivity of  $P(\mu)$  allows us to conclude that $P \simeq P(\mu)$.
     \end{proof}
    
    The following proposition extends Proposition \ref{description-of-simple} to indecomposable modules. 
 \begin{proposition}\label{indecomposable-Cl}

        (a)        If $X$ is an indecomposable  module  from $\cW_{D(a|b)}$, then $X$ is isomorphic to $X^+\otimes (X^-)^\dagger$ for some indecomposable module $D(a|0)$-module
        $X^+$ from $\cW_{D(a|0)}$ and some simple module $X^-$ from $\cW_{Cl(b|0)}$. 
        
        (b) If $X$ is an indecomposable module from  $\cW_{Cl(a|b)}$, then $X$ is isomorphic to $((X^+)^\dagger\otimes (X^-)^\dagger)^\dagger$ for some
        simple module $X^+$ from $\cW_{Cl(a|0)}$
        and some indecomposable 
module $X^-$  from $\cW_{D(b|0)}$. 
      \end{proposition}
      \begin{proof} Let us prove (a). Set $A = D(a|b)$. The indecomposability of  $X$ implies $\supp X\subset\mu+Q_A$ for some $\mu \in \gh_A^{\vee}$. If
        $S$ and $S'$ are simple subquotients of $X$ then $\supp S\subset \supp S'+Q_A$, and therefore $S$ and $S'$ have the same support when restricted to
        $D(0|b)$. This, together with Proposition \ref{description-of-simple}(a),
        implies the existence of isomorphisms $S\simeq Y\otimes (X^-)^\dagger$ and $S'\simeq Z\otimes (X^-)^\dagger$ for some simple   module $X^-\in\cW_{Cl(b|0)}$ and some simple modules   $Y,Z\in\cW_{D(a|0)}$.  Moreover, according to Lemma \ref{semisimplicity},  the restriction of $X$ to $D(0|b)$ is a semisimple $D(0|b)$-module. Hence this restriction is isomorphic to an isotypic component of the simple $D(0|b)$-module $(X^-)^\dagger$. This allows us to conclude that the map 
        $$\Hom_{D(0|b)}((X^-)^\dagger,X) \otimes (X^-)^\dagger \to X$$
        is an isomorphism. 
        
        Therefore we can set $X^+ :=\Hom_{D(0|b)}((X^-)^\dagger,X)$. Finally, the  indecomposability of $X$ implies the indecomposability of $X^+$.

        The proof of (b) is similar, but instead of Proposition \ref{description-of-simple}(a) one uses Proposition \ref{description-of-simple}(b). 
      \end{proof}
      
 \begin{corollary}\label{blocks-Cl}
            (a) If $b<\infty$ then the category $\cW_{D(a|b)}$
is equivalent to the category $\cW_{D(a|0)}$.
The category $\cW_{D(a|\infty)}$ decomposes into a direct product of subcategories
   $\mathcal W_{[\mathbb A]}$ where $[{\mathbb A}]$ runs over equivalence classes of subsets of $\mathbb Z_{>0}$ as in Lemma \ref{Clifford},   and each subcategory $\mathcal W_{[\mathbb A]}$ is equivalent to the category $\cW_{D(a|0)}$.
   
     (b)    If $a<\infty$ then the category $\cW_{Cl(a|b)}$ is equivalent to the category  $\cW_{D(b|0)}$.
The category $\cW_{Cl(\infty | b)}$ decomposes into a direct product of subcategories
   $\mathcal W_{[\mathbb A]}$ where $[{\mathbb A}]$ runs over  equivalence classes of subsets of $\mathbb Z_{>0}$  as in Lemma \ref{Clifford},  and each subcategory $\mathcal W_{[\mathbb A]}$ is equivalent to the category $\cW_{D(b|0)}$. 
   
     (c) Every block of $\cW_{D(a|b)}$ and of $\cW_{Cl(a|b)}$ is equivalent to the block $\cW^\Gamma_{D(c|0)}$ of $\cW_{D(c|0)}$
     for some $c\leq\infty$  and  $\Gamma=Q_{D(c|0)}$ .

   (d) Two blocks $\mathfrak B_1$ and $\mathfrak B_2$ of $\cW_{D(c|0)}$
   are equivalent if and only if
     $c(\mathfrak B_1)=c(\mathfrak B_2)$ where $c(\mathfrak B)$ denotes the cardinality of the set of
   isomorphism classes of simple objects in $\mathfrak B$.
    
 \end{corollary}
\begin{proof} Again we  prove just (a) since (b) is similar.  Let $X^-$ be a preferred simple $Cl(b|0)$-module and $\mathcal W_A(X^-)$ be the full subcategory of $\mathcal W_A$ with objects of
     the form $X^+\otimes  (X^-)^\dagger$ for preferred weight $D(a|0)$-modules $X^+$. It follows from Proposition \ref{indecomposable-Cl} that ${\mathcal W}_A$ is the direct product of its subcategories
     ${\mathcal W}_A (X^-)$ where $X^-$ runs over the set of isomorphism classes of $Cl(b|0)$-modules.  Each category $\mathcal W_A(X^-)$ is equivalent to the category of preferred weight
     $D(a|0)$-modules via the functors $\cdot\otimes (X^-)^\dagger$ and $\Hom_{D(0|b)}((X^-)^\dagger,\cdot)$. If $b<\infty$,  there is a single isomorphism class to which  $X^-$ belongs. If $b=\infty$,
     the isomorphism classes of modules in $\mathcal W_{Cl(b|0)}$
     are enumerated by the equivalence classes of subsets of $\mathbb Z_{>0}$ from Lemma \ref{Clifford}.
     
       Parts  (c) and (d) follow from parts (a) and (b) and from the classification of  blocks in $\cW_{D(c|0)}$ for $c < \infty$ in \cite{GrS}, and in $\cW_{D(\infty|0)}$ in \cite{FGM}.
        \end{proof}

  We conclude this section by a structural result on indecomposable weight $A$-modules with finite-dimensional weight spaces.  
                \begin{theorem}\label{filtration} Any indecomposable  $A$-module $X$ in $\mathcal W_A$ with finite-dimensional weight spaces has a strict $A$-module filtration $X =\cup_{n\in R}X_n$ (i.e.,  $X_n \subsetneq X_m$ for  $n <m$)  for some interval $R$ in $\mathbb Z$, satisfying $\cap_{n\in R}X_n = \left\{0 \right\}$ and such that  $X_n/X_{n-1}$ is a simple $A$-module for any $n,n-1 \in R$. 
                \end{theorem}
                \begin{proof} Due to Corollary \ref{blocks-Cl} we can reduce this statement to the case $A=D(a|0)$. If $a$ is finite then $X$  has finite length and the statement
                  is trivial. For any $a$, a preferred simple weight $D(a|0)$-module is  determined up to isomorphism by its support. Therefore, if  $X$ belongs to a block $\mathfrak B$ with $c(\mathfrak B)<\infty$, the statement is trivial since $X$ necessarily has finite length. 
                  
             We can thus assume that $X$ belongs to a block $\mathfrak B$ with $c(\mathfrak B)=\infty$, and by {Corollary  \ref{blocks-Cl} (c)} we can assume further that
                  $\Gamma=Q_{A}$.
                  Then, simple objects in $\mathfrak B$ are enumerated (up to isomorphism) by finite subsets $\mathbb A$ of $\mathbb Z_{>0}$. For a subset $\mathbb A$, we set $\zeta_{\mathbb A}:=-\sum_{i\in \mathbb A}\zeta_{i}$
                  and choose a basis $\{v_i^{\mathbb A}\}$ of the weight subspace $X^{\zeta_{\mathbb A}}$. Let $U$ be the union of these bases. Note that every cyclic  $A$-module is multiplicity free and
                  has at most countably many cyclic submodules generated by vectors of weights of the form $\zeta_{\mathbb A}$. Consider the set $\mathcal X$ 
                  of cyclic submodules of $X$ consisting of all modules $Au$ for $u\in U$ and all cyclic submodules of $Au$ generated by weight vectors (the weights necessarily having
                   the form $\zeta_{\mathbb B}$ for finite subsets ${\mathbb B}$ of $\mathbb Z_{>0}$). Then $\mathcal X$ is a partially ordered set with respect to the inclusion order.
              Clearly, $X=\sum_{Y\in \mathcal X}Y$.
         
        We claim that any interval in this partial order is finite.
        To prove this, it suffices to consider an interval of the form $[Av,Aw]$ where $Av\subset Aw$. Let $v\in X^{\zeta_{\mathbb A}}$ and $w\in X^{\zeta_{\mathbb B}}$ for some finite sets ${\mathbb A},{\mathbb B}$. Note that $Aw$ is a quotient of the indecomposable projective $A$-module $P(\zeta_{\mathbb B})$. Therefore $$v=d\prod_{i\in I}x_i\prod_{j\in J}\partial_jw$$ for some $d\in \mathbb C^*$ and some finite subsets
        $J\subset \mathbb A$, $I\subset \mathbb Z_{>0}\setminus \mathbb B$. If $Av\subset Au\subset Aw$ then
        $$u=d'\prod_{i\in I'}x_i\prod_{j\in J'}\partial_jw,\ v=d''\prod_{i\in I''}x_i\prod_{j\in J''}\partial_jv.$$
        Note that  $I=I'\sqcup I''$, $J=J'\sqcup J''$. Since for fixed $I,J$ there exist finitely many choices
        for $I',I'',J',J''$, the claim follows.

Recall that by the Szpilrajn theorem \cite{Mar} any partial order can be extended to a total order.            
Moreover, we claim that any interval-finite partial order on a countable set $I$ can be extended to an  interval-finite total order.  Indeed, assume that $I$ does not have a smallest or greatest element.
(If $I$ is bounded above or below, the proof is similar). We can choose a sequence of distinct elements $\{x_i\mid i\in\mathbb Z\}$ such that if
                    $x_i<x_j$ then $i<j$, and also $I=\cup [x_i,x_{i+1}]$. Let $U_n=\cup_{i=-n}^{n-1}[x_i,x_{i+1}]$ for $n>0$. Using induction we can define a total order on $U_n$ as required. Indeed, one can see that
                    $U_{n+1}\setminus U_{n}=Y\cup Z$ where all elements of $Z$ are not less than elements of $U_n$ and all elements of $Y$ are not greater than the elements of $U_n$. On the other hand, both  $Y$ and $Z$ are finite and therefore
                    one can clearly define a suitable total order on them.
             
                    This argument endows $\mathcal X$ with a total order $\prec$ such that the ordered set $(\mathcal X,\prec)$ is isomorphic to $(\mathbb Z, <)$, $(\mathbb Z_{<0}, <)$, $(\mathbb Z_{>0}, <)$, or some finite interval of $\mathbb Z$.  We enumerate the elements of $\mathcal X$ using this isomorphism.
                 Set                  $X_n:=\sum_{i<n} Y_i$ for $Y_i \in \mathcal X$.
                 Let us prove that the $A$-module $X_n/X_{n-1}$ is simple for any $n$. Indeed, $X_n/X_{n-1}\simeq Y_n/(Y_n\cap X_{n-1})$. Since $Y_n\cap X_{n-1}$ contains all
                 proper cyclic submodules of $Y_n$, the submodule $Y_n\cap X_{n-1}$ is the unique maximal submodule of
                 $Y_n$ and the quotient $Y_n/(Y_n\cap X_{n-1})$
                 is simple. If $(\mathcal X,\prec)$ has no minimal element then clearly $\cap_n X_n=\{0\}$. If $X_1=Y_1$ is the minimal element of
                 $(\mathcal X,\prec)$, then $X_1$ is simple and we add $X_0:=\{0\}$.
               \end{proof}
               \begin{example} 
                 Let  $A=D(\infty|0)$, $\mu\in Q_A$, and let $X$ be an indecomposable $A$-module of infinite length  with finite weight multiplicities.
                 
                 \begin{enumerate}
                 \item[(a)] Theorem \ref{filtration} implies that $X$ admits
                 a $\mathbb Z_{>0}$-filtration with simple quotients whenever $X$ has a simple submodule contained in any nonzero submodule of $X$.
                 Therefore the $A$-module $F(\mu)$ has such a filtration by Lemma \ref{indecomposable}.
               \item[(b)] Similarly, if $X$ has a unique maximal submodule then $X$ admits a $\mathbb Z_{<0}$-filtration with simple quotients.
                 In particular, this applies to the  $A$-module $P(\mu)$.
               \item[(c)] Fix an isomorphism $A\simeq A\otimes A$ of associative algebras and consider
                 $X:=F(\mu)\otimes P(\mu)$ as an $A$-module via this isomorphism. One can see that
                 $X$ has neither a simple submodule nor a simple quotient. Nevertheless, by  Theorem \ref{filtration} the module $X$ admits
                 a $\mathbb Z$-filtration with simple quotients.
                   \end{enumerate}
                 \end{example}

  \subsection{Weight modules over $A_{ev}$ and  $A_{\bar{0}}$ for $A=D(a|b)$ or $A= Cl(a|b)$}\label{evenpart}
Let $\tau: Q_A\to \mathbb \mathbb Z/2\mathbb Z$ be a surjective homomorphism of abelian groups. We define
    $$B:=\bigoplus_{\tau(\mu)=\bar 0}A^\mu,\quad B':=\bigoplus_{\tau(\mu)=\bar 1}A^\mu.$$
    Then $B$ is a subsuperalgebra of $A$ containing $H_A$, and the decomposition $A=B\oplus B'$ defines a $\mathbb Z/2\mathbb Z$-grading.
    
In this subsection we establish an equivalence between the category
${\cW}_A$ and the category $\cW_{B}$ of preferred weight $B$-modules. This result applies to the particular cases  $B=A_{\bar{0}}$ and $B=A_{ev}$.
(For $B = A_{\bar{0}}$  preferred $B$-modules are purely even $B$-modules.)

  The root lattice $Q_{B}=\ker\tau$ is an index-two subgroup in $Q_A$. Consider the block ${\cW}_A^\Gamma$ for some equivalence
  class $\Gamma\subset \gh^\vee_A$. Note that $\Gamma=\Gamma'\sqcup \Gamma''$, where
  $\Gamma':=(\mu+Q_{B})\cap\Gamma$ for some $\mu\in\Gamma$. This decomposition depends on the choice of $\mu$ but only up to swapping $\Gamma'$ and
  $\Gamma''$. By $\cW_{B}^{\Gamma'}$ we denote the subcategory of $\cW_{B}$ of $B$-modules with support in $\Gamma'$.

  \begin{theorem}\label{even} The abelian categories  ${\cW}_A^\Gamma$ and $\cW_{B}^{\Gamma'}$ are equivalent.
  \end{theorem}
  \begin{proof} We define functors $R:{\cW}_A^\Gamma \to \cW_{B}^{\Gamma'}$ and $I:\cW_{B}^{\Gamma'}\to{\cW}_A^\Gamma$ by setting
    $$R(X):=\bigoplus_{\mu\in\Gamma'} X^\mu,\quad I(Y):=A\otimes_{B}Y.$$
    We observe that $R$ is exact,  $I$ is right exact, and $I$ is left adjoint to $R$. Therefore, there are  canonical morphisms of functors $\phi:IR\to \operatorname{Id}_{\cW_{A}}$ and $\psi:\operatorname{Id}_{\cW_{B}}\to RI$.
    It remains to check that both functors are isomorphisms on objects. Recall that for any $\mu\in \Gamma$ the induced module $P(\mu)=A\otimes_{H_A} \left( \Pi^{p(\mu)} \C_{\mu} \right)$ is projective in $\cW_A$. Similarly, the $B$-module
    $Q(\mu):=B\otimes_{H_A} \left( \Pi^{p(\mu)} \C_{\mu} \right)$ is projective in $\cW_{B}$. By construction we have $I(Q(\mu))\simeq P(\mu)$ and $R(P(\mu))\simeq Q(\mu)$. Thus $\phi(P(\mu))\simeq P(\mu)$ and $\psi(Q(\mu))\simeq Q(\mu)$.
   Every object in $\cW_A$ (respectively, $\cW_B$) has a resolution with terms given by direct sums of $P(\mu)$-s (respectively, $Q(\mu)$-s). Hence $\phi$ and $\psi$ are isomorphisms on  objects.
         \end{proof}

    \subsection{Weight modules over $A_0$} Here we classify simple bounded $A_0$-modules. 
    We note that  $\gh_A\subset A_0$ and that the root lattice $Q_{A_0}$ is the sublattice of $Q_A$ generated by $\zeta_i-\zeta_j$ for  $i,j\neq 0$, $i \neq j$.  As before, we can work with preferred modules only.
    We introduce a new equivalence relation on $\gh^\vee$ by setting $\mu\approx_{0}\nu$
    iff $\mu\approx\nu$ and $\mu-\nu\in Q_{A_0}$.

    \begin{theorem}\label{Zero}
      (a) For every $\mu\in \gh_A^\vee$ there exists a unique (up to isomorphism) preferred simple weight module $Y(\mu)$ such that $\mu\in\supp Y(\mu)$.

      (b) Two simple preferred $A_0$-modules $Y(\mu)$ and $Y(\nu)$ are isomorphic if and only if  $\mu\approx_{0}\nu$.
      \end{theorem}
      \begin{proof} (a) We define the $A_0$-module $Y(\mu)$ to be the $A_0$-submodule of $X(\mu)$ generated by the weight space $X(\mu)^\mu$. It is simple
        since  for every proper submodule $Z$ of $Y(\mu)$ we have $AZ\cap Y(\mu)=Z$. Furthermore any simple weight $A_0$-module, whose 
        support contains $\mu$, is isomorphic to the unique simple quotient of the induced module $A_0\otimes_{H_A}\mathbb C_\mu$. This proves (a).
        
        (b) It follows from (a) that $Y(\mu)$ and $Y(\nu)$ are isomorphic if and only if $\mu\in\supp Y(\nu)$. On the other hand,
        $$\supp Y(\nu)=\supp X(\nu)\cap (\nu+Q_{A_0}).$$
      This implies the statement.
        \end{proof}
    Let $A=D(a|b)$. Any $A_0$-module $M$ is also a module over $\operatorname{Lie}A_0$, the Lie superalgebra associated to $A_0$.  We call $M$
    {\it integrable} if $M$ is integrable as an $\mathfrak{sl}(a|b)$-module.
    \begin{proposition}\label{integrable}

      (a) A simple  weight $A_0$-module $Y(\mu)$ is integrable if and only if $\mu_i\in\mathbb Z_{\geq 0}$
      for all $i>0$ or $\mu_i\in\mathbb Z_{<0}$ for all $i>0$.

      (b) Every simple weight $A_0$-module is integrable as a  $D(0|b)_0$-module.
    \end{proposition}
    \begin{proof} (a) By a direct inspection of $\supp Y(\mu)$ one sees that, if $\mu$ satisfies the condition of the proposition, then any
      $\nu\in\supp Y(\mu)$
      satisfies the same
      condition. Therefore the set $(\nu+\mathbb Z\alpha)\cap\supp Y(\mu)$ is finite for any    $\nu\in\supp Y(\mu)$ and any root $\alpha$ of
      $\mathfrak{sl}(a|b)$.
    This implies that $Y(\mu)$ is integrable whenever $\mu$ satisfies the condition of the proposition.

      On the other hand, if there exist $i,j>0$, $i\neq j$, such that $\mu_j$ is not an integer or
      $\mu_i\in\mathbb Z_{\geq 0}$ and $\mu_j\in\mathbb Z_{<0}$, then $x_i\partial_j$ acts freely on $Y(\mu)^\mu$.

      (b) For any $\mu\in\gh^\vee_A$ and any $\alpha=\zeta_i-\zeta_j$ for $i,j<0$, at most one of $\mu+\alpha$ and $\mu-\alpha$ lies in $\gh^\vee_A$.
      Since the support of any weight $A_0$-module is a subset of  $\gh^\vee_A$, the statement follows.
      \end{proof}

      \begin{proposition}\label{faithful} Suppose $A=D(\infty|\infty)$. A simple $A_0$-module $Y(\mu)$ is faithful if and only if the set of values $$S_i=\{\nu_i\mid \nu\in \supp Y(\mu)\}$$ is infinite 
        at least for one $i$.
        \end{proposition}
        \begin{proof} We first note that (by definition) $\nu_i\in\{0,1\}$ for all $\nu\in \supp Y(\mu)$ and $i<0$. If $i>0$ and $\nu_i\geq m$ for some $m\in \mathbb Z_{>0}$, then $(x_j\partial_i)^mY(\mu)^{\nu}\neq 0$ for any $j>0$.
          Similarly, if $i>0$ and $\nu_i<-m$ for some  $m\in \mathbb Z_{>0}$, then $(x_i\partial_j)^mY(\mu)^{\nu}\neq 0$ for any $j>0$. Thus the sets $S_i$ are infinite for all $i>0$ whenever $S_i$ is infinite for some $i>0$.
                    Moreover, we observe that the set $T_k:=\{(\nu_1,\dots,\nu_k)\mid\nu\in\supp Y(\mu)\}$ is Zariski dense in $\C^k$ for every $k>0$.

          Next we notice that if $S_i$ is finite for some $i>0$ then $\prod_{s\in S_i}(u_i-s)\in\Ann_{A_0}Y(\mu)$, where $u_i=x_i\partial_i$.

          It remains to show that if $S_i$ is infinite for any positive $i$ then $\Ann_{A_0}Y(\mu)=0$. Clearly, $\Ann_{A_0}Y(\mu)$ is a weight $\gh$-module with respect to the adjoint action of $\gh$. Furthermore, for any
          $u\in A_0^{\gamma}$ there exists $v\in A_0^{-\gamma}$ such that $uv\neq 0$. Thus, it suffices to prove that $\Ann_{A_0}Y(\mu)\cap H_A=\{0\}$. Any $u\in H_A$ can be written in the form
          $$u=p_0(u_1,\dots,u_k)+\sum_{i=1}^lp_i(u_1,\dots,u_k)u_{-i}$$
          for some $k,l \in \mathbb Z_{>0}$ and polynomials $p_0,p_1,\dots,p_l$. The condition $u\in \Ann_{A_0}Y(\mu)$ implies
          $p_i(\nu_1,\dots,\nu_k)=0$ for all $\nu\in\supp Y(\mu)$ and $i=1,\dots,l$. Therefore $p_i(T_k)=0$, and hence $u=0$.
          \end{proof}
   
      \begin{corollary}\label{primcor} The ideals $\ker\Upsilon^\pm$ are primitive ideals of $U(\mathfrak{sl}(\infty|\infty))$.
          \end{corollary}

\section{Classification of simple bounded weight $\osp$-modules at infinity}

We are now ready to describe the category of bounded weight $\gg$-modules for $\gg = \osp(2a|2b), \osp(2a+1|2b)$. In what follows we assume that $\gg$ is infinite dimensional, i.e., that at least one of $a,b$ equals $\infty$.   We fix an exhaustion of $\gg$ as $\limarr \gg_k$, where $ \gg_k \simeq \osp (2a_k| 2b_k)$ or $ \gg_k \simeq \osp (2a_k+1| 2b_k)$, and $a_k,b_k \in \Z_{>0}$ satisfy $a_k = a$ for $a < \infty$ and  $b_k = b$ for $b < \infty$. 

We start with the following observation.

         \begin{proposition}\label{tosp}  If $M$ is a bounded $\gg$-module, then the restriction of $M$ to  $\so(2a)$ or  $\so(2a+1)$ is integrable and semisimple. 
                  \end{proposition}
         \begin{proof}  $M$ is a bounded semisimple $\gh$-module, and hence $M$ is a bounded weight $(\so(2a)+\gh)$- or $(\so(2a+1)+\gh)$-module. Therefore, as an $\so(2a)$- or $\so(2a+1)$-module, $M$ is isomorphic to a direct sum of bounded weight $\so(2a)$- or $\so(2a+1)$-modules. As mentioned in Section 1, a bounded  weight $\so(2a)$- or $\so(2a+1)$-module is integrable for $a=\infty$, and is a sum of finite-dimensional modules for $a < \infty$. Therefore the semisimplicity claim holds trivially for $a< \infty$. For  $a = \infty$ the semisimplicity claim follows from Theorem 3.7 in \cite{PSer2}.
         \end{proof}
         Recall that an \emph{odd reflection} is the replacement of a Borel subsuperalgebra $\mathfrak b$ of $\gg$ by 
a Borel subsuperalgebra $\mathfrak b'$ of $\gg$ such that exactly one odd root $\alpha$ of $\mathfrak b$  is not a root of $\mathfrak b'$ (and hence $-\alpha$ is a root of $\mathfrak b'$). If $L_{\mathfrak b} (\lambda)$ denotes an irreducible $\gg$-module with $\mathfrak b$-highest weight $\lambda$ and purely even highest-weight vector, then  $L_{\mathfrak b} (\lambda)$ is isomorphic either to $L_{\mathfrak b'} (\lambda)$ or to  $\Pi L_{\mathfrak b'} (\lambda - \alpha)$. The latter case, called a \emph{typical} reflection, occurs precisely when $(\lambda , \alpha)\neq 0$, while the former case, called an \emph{atypical} reflection, occurs when $(\lambda, \alpha) =  0$. 
         
By $J_\gg$ we denote the kernel of $\Psi_{a|b}$ if $\gg = \osp(2a|2b)$, and respectively of $\Theta_{a|b}$ if $\gg = \osp(2a+1|2b)$.    Recall that $J_{\gg}$ is the annihilator of any spinor-oscillator representation. Moreover, it is obvious that $J_{\gg}=\varinjlim J_{\gg_k}$ whenever $\gg=\varinjlim \gg_k$ for an inductive system of  finite-dimensional Lie superalgebras $ \gg_k$ of type $\mathfrak{osp}$.

  \begin{lemma}\label{boundedideal} Let $\gq=\osp(m|2n)$ for $m, n \in {\mathbb Z}_{\geq 0}$, and $I\subset U(\gq)$ be a bounded primitive ideal of degree $d$. Assume that at least one of the simple ideals of $\gq_{\bar{0}}$ has rank greater than $d$. Then $d=1$. Moreover $I=J_{\mathfrak{q}}$, unless $I$  is the augmentation ideal or  the annihilator of a defining module. 
  \end{lemma}
\begin{proof} For $m \leq 1$ the statement follows directly from Lemmas  \ref{osp12}  and \ref{finiterank}. Therefore in the rest of the proof we assume that $m\geq 2$.

 By Musson's Theorem \cite{Musson}, $I=\Ann_{U(\gq)} L_{\gb}(\lambda)$ for some Borel subsuperalgebra $\gb$ and some weight $\lambda$. For $\lambda=0$, the ideal $I$ is the augmentation ideal. For the rest of the proof we assume $\lambda\neq 0$. Let $\gs$ be a simple ideal of $\gq_{\bar{0}}$ of rank greater than $d+1$. We can choose the Borel subsuperalgebra $\gb$ so that 
	its base of simple roots contains a base of simple roots for
  $\gs$.    By $\mu$ we denote the weight of $\gs$ obtained from $\lambda$ by restriction.
  
 In order to study the annihilator $I$ of the simple highest weight $\gq$-module $L_{\gb}(\lambda)$, we will consider $L_{\gb}(\lambda)$ as a highest weight module over a variable Borel subalgebra $\gb'$ obtained from $\gb$ by some sequence of odd reflections. Then $\lambda'$ will denote the corresponding
  highest weight, and $\mu'$ will be its restriction to $\gs$. Lemma \ref{degreefinite} implies that the simple $\gs$-modules with highest weights $\mu$ and  $\mu'$ are necessarily multiplicity free.
  
  We may assume that $\gb'$ is obtained from $\gb$ by  odd reflections with respect to some isotropic odd  roots $\alpha_1,\dots,\alpha_r$.  It is essential to note that there are at most four nonisomorphic multiplicity free simple  weight $\gs$-modules which have a highest weight with respect to a fixed Borel subalgebra  of $\gs$. (Indeed, these are the trivial, natural, and spinor modules for $\gs \simeq \so (m)$, and the trivial, natural, and oscillator modules for $\gs \simeq \sp (2n)$.)
  This shows that each of the weights $\mu$ and $\mu'$ can take at most four different values. Moreover, since $\lambda,\lambda'$  have the same image modulo the root lattice of $\gq$, it is easy to check that for a given $\mu$ there is a unique $\mu'$ with $\mu'\neq \mu$. Therefore in a shortest chain of odd reflections connecting $\gb$ and $\gb'$ there can be at most one typical reflection. 
 
  Assume $\gs=\mathfrak{sp}(2n)$. If $m=2\ell+1$  we fix the simple roots
  $$\varepsilon_1-\varepsilon_2,\dots,\varepsilon_\ell-\delta_1,\dots, \delta_{n-1}-\delta_n,\delta_n,$$
  and if $m=2\ell$ we take the simple roots
  $$\varepsilon_1-\varepsilon_2,\dots,\varepsilon_\ell-\delta_1,\dots, \delta_{n-1}-\delta_n,2\delta_n.$$
Set $\lambda=a_1\varepsilon_1+\dots +a_\ell\varepsilon_\ell+\mu$.  The  above conditions and Lemma \ref{finiterank} show that for $\mu \neq \mu'$ one of the following holds:

  \begin{enumerate}
  \item $\mu=0$, $\mu'=\delta_1$,
    \item $\mu=\delta_1$, $\mu'=0$,
\item $\mu=-\frac{1}{2}(\delta_1+\dots+\delta_{n})$, $\mu'=-\frac{1}{2}(\delta_1+\dots+\delta_{n-1})-\frac{3}{2}\delta_n$,
    \item $\mu=-\frac{1}{2}(\delta_1+\dots+\delta_{n-1})-\frac{3}{2}\delta_n$, $\mu'=-\frac{1}{2}(\delta_1+\dots+\delta_{n})$.
  
  \end{enumerate}

  Consider the first case. We start by applying the odd reflections corresponding to the sequence of odd roots $\varepsilon_\ell - \delta_1,\dots,\varepsilon_1-\delta_1$. Since $\lambda \neq 0$, exactly  one of these reflections must be typical,
  say with respect to $\varepsilon_p-\delta_1$. This implies $a_{p+1}=\dots=a_{\ell}=0$, $a_1=\dots=a_{p-1}=-1$.  Next, an application of the reflections corresponding to $\varepsilon_\ell - \delta_2,\dots,\varepsilon_1 - \delta_2$ cannot change $\lambda'$. This is only possible
  for $p=1$ and $\lambda=\varepsilon_1$, and then $L_{\gb}(\lambda)$ is a defining representation.

 Let us deal with the second case. The odd reflections with respect to the roots $\varepsilon_\ell - \delta_1,\dots,\varepsilon_1-\delta_1$ do not change $\lambda$, i.e., they are all atypical. Therefore $a_1=\dots=a_\ell=-1$, but then the reflection with respect to $\varepsilon_\ell - \delta_2$ is typical and
  $\mu'=\delta_1+\delta_2$. This proves that the second case is impossible.

  Now, consider the third case. Here we perform in some order all odd reflections with roots $\varepsilon_i  - \delta_j$,  $i=1,\dots,\ell$,  $j=1,\dots,n-1$, and check that all these reflections do not change $\lambda$. This forces
  $a_1=\dots=a_\ell=\frac{1}{2}$. Hence $\lambda = \frac{1}{2} \left(\varepsilon_1+\cdots +\varepsilon_\ell\right)  -\frac{1}{2}(\delta_1+\dots+\delta_{n})$ and  
  $L_{\gb}(\lambda)$ is a spinor-oscillator representation.

  Finally, let us look at the fourth case. We can show that $a_1=\dots=a_\ell=\frac{1}{2}$ in the same way as in the third case. Therefore, if $m$ is even we have $\lambda = \frac{1}{2} \left(\varepsilon_1+\cdots +\varepsilon_\ell\right)  -\frac{1}{2}(\delta_1+\dots+\delta_{n-1})-\frac{3}{2}\delta_n$, and 
   $L_{\gb}(\lambda)$ is a spinor-oscillator representation not isomorphic up to parity change to a spinor-oscillator representation that occurred in the third case. 
  If $m$ is odd, then by Lemma \ref{osp12} the restriction of  $\lambda$  to $\mathfrak{osp}(1|2n)$ with roots $\pm \delta_i \pm \delta_j, \delta_r - \delta_s, \pm \delta_i$ for $r \neq s$,
     must equal $-\frac{1}{2}(\delta_1+\dots+\delta_{n})$. This contradicts our assumption for $\mu$,  therefore the fourth case forces $m$ to be even. 
  
  This proves our claim for $\mathfrak s = \sp (2n)$ since in case (1) $I$ is the annihilator of a defining representation, while in cases (3) and (4) $I$ is the annihilator of a spinor-oscillator representation.
  
We conclude the proof by essentially repeating the above argument  for $\gs=\mathfrak{o}(m)$.
For  $m=2\ell$ we fix the simple roots
  $$\delta_1-\delta_2,\dots,\delta_{n-1}-\delta_n,\delta_n-\varepsilon_1,\varepsilon_1-\varepsilon_2,\dots,\varepsilon_{\ell-1}-\varepsilon_\ell,\varepsilon_{\ell-1}+\varepsilon_\ell,$$
and for $m=2\ell+1$ we choose  the simple roots
 $$\delta_1-\delta_2,\dots,\delta_{n-1}-\delta_n,\delta_n-\varepsilon_1,\varepsilon_1-\varepsilon_2,\dots,\varepsilon_{\ell-1}-\varepsilon_\ell,\varepsilon_\ell.$$
A priori there are the following cases for $\mu \neq \mu'$:
 \begin{enumerate}
  \item $\mu=0$, $\mu'=\varepsilon_1$,
    \item $\mu=\varepsilon_1$, $\mu'=0$,
\item $\mu=\frac{1}{2}(\varepsilon_1+\dots+\varepsilon_{\ell})$, $\mu'=\frac{1}{2}(\varepsilon_1+\dots+\varepsilon_{\ell-1}-\varepsilon_{\ell})$,
\item  $\mu=\frac{1}{2}(\varepsilon_1+\dots+\varepsilon_{\ell-1}-\varepsilon_{\ell})$, $\mu'=\frac{1}{2}(\varepsilon_1+\dots+\varepsilon_{\ell})$.
\end{enumerate}
All these cases can be treated in the same way as above.  \end{proof}

  \begin{corollary} Let  $\gq$ and $I$ are as in the previous lemma. Then the superalgebra of $\gh$-invariants $\left(U(\gq)/I\right)^{\gh}$ is abelian. Hence any weight $\gq$-module generated by a single weight vector is multiplicity free.
  \end{corollary}

We are now ready to prove the following.
  
  \begin{proposition}\label{limit} Let $M$ be a simple bounded $\gg$-module. Then $M$ is multiplicity free. Moreover, $M$ satisfies $\Ann_{U(\gg)} M=J_\gg$ or $M$ is a trivial or a natural module.  \end{proposition}
  \begin{proof} Let $I:=\Ann M$, $\bar U:=U(\gg)/I$, and let $\lambda$ be a weight of $M$. Then a standard argument shows that $M^\lambda$ is a simple $\bar U^{\gh}$-module. Next, set
    $$\gh_k:=\gg_k\cap \gh,\quad \bar U_k:=U(\gg_k)/(U(\gg_k)\cap I).$$
    We have $\bar U^{\gh}=\varinjlim \bar U_k^{\gh_k}$.  Since $\gg$ is infinite dimensional, Lemma \ref{boundedideal} implies that for sufficiently large $k$ the simple $\bar U_k^{\gh_k}$-constituents of the module $M^\lambda$ are one-dimensional.  
    By passing to  the direct limit we obtain $\dim M^\lambda=1$.  Furthermore, again by Lemma  \ref{boundedideal} we see that  the annihilator of $U(\gg_k)M^\lambda$ equals $J_{\gg_k}$, unless $U(\gg_k)M^\lambda$ is 
    a trivial representation or a defining representation. The statement follows by passing to the direct limit.  
    \end{proof}
    \begin{remark}
The claim of Proposition \ref{limit} is proved in \cite{GrP} in the case where $\gg = \gg_{\bar{0}}$, i.e., for $\gg = \sp(\infty), \so(\infty)$.
    \hfill${\bigcirc}$    \end{remark}
    
    We say that a simple weight $\gg$-module $M$ is of {\it 
    spinor-oscillator type} if it is annihilated by $J_{\gg}$, i.e.,  $M$ is   obtained by pullback along the homomorphism $\Theta_{a,b}$  from a weight $Cl(a|b)$-module or, respectively, along the homomorphism  $\Psi_{a,b}$ from a weight $Cl(a|b)_{ev}$-module. Proposition \ref{limit} implies the following.

    \begin{corollary}\label{cor-osc} Let $M$ be a simple bounded weight $\gg$-module such that $M\not \simeq V, \Pi V,\mathbb C, \Pi \mathbb C$. Then $M$ is  of spinor-oscillator type. 
    \end{corollary}

 Note that 
every simple weight $\sp(2b)$-module $T$ of oscillator type (as defined in Section 1) is the pullback of a (unique, up to isomorphism) simple weight $Cl(0|b)_{ev}$-module $\tilde T$. This follows from the fact that the  ideal $\ker \Psi_{0,b}$ of $U (\sp(2b))$ is the primitive ideal not only of the oscillator representations but of any simple multiplicity free weight module of $\sp (2b)$. For  $b < \infty$ this is well known, and for $b=\infty$ see \cite{GrP}.

	 Given $T$ as above, the module $\tilde T$ generates a unique simple weight $Cl(0|b)$-module which has the form $\tilde T \oplus \tilde T'$ as a $Cl(0|b)_{ev}$-module. The pullback of $\tilde T'$ to $\sp(2b)$ is by definition the \emph{twin} of $T$ and is a simple module.  Similarly, any spinor $\so(2a)$-module is the pullback of a simple weight $Cl(a|0)_{ev}$-module, and 
	 we call two  spinor $\so(2a)$-modules {\emph{twins}} if they are  pullbacks of the two simple $Cl(a|0)_{ev}$-constituents of a simple $Cl(a|0)$-module. For $\so(2a+1)$ we declare two  spinor $\so(2a+1)$-modules  to be {\emph{twins}} if they are isomorphic. 

 We are ready to state our explicit description of  simple bounded weight $\gg$-modules. 
 	 
	 \begin{theorem}  \label{thm-m-twins} Let $M$ be a simple bounded weight $\gg$-module of spinor-oscillator type. Then the following statements hold.
	 \begin{itemize}
\item[(a)]	$M_{\bar{0}}$ and  $M_{\bar{1}}$ are simple $\gg_{\bar{0}}$-modules.
\item[(b)]	There exist twin spinor $\so(2a)$- or $\so(2a+1)$-modules $S$ and $S'$, and twin simple  $\sp(2b)$-modules $T$ and $T'$ of oscillator type, such that 
                      \begin{equation} \label{eq-m-s-n}
                      M_{\bar{0}}\simeq S\otimes T,\quad M_{\bar{1}}\simeq \Pi (S'\otimes T').
                      \end{equation}
                      The modules $S,S',T,T'$ are unique up to isomorphism and determine $M$ up to isomorphism.
\item[(c)] Any pair $(S,T)$ where  $S$ is a  spinor $\so(2a)$- or $\so(2a+1)$-module and $T$ is a simple  $\sp(2b)$-module of oscillator type determines a simple bounded weight $\gg$-module $M$ of spinor-oscillator type for which \eqref{eq-m-s-n} holds. 
	 \end{itemize}	 
	 \end{theorem}
	 \begin{proof} Let $A=Cl(a|b)$.
          Claim (a) follows directly from Lemma \ref{associative} since the map  $\Psi_{a|b}:U(\gg_{\bar 0})\to (A_{ev})_{\bar 0}$ (respectively, $\Theta_{a|b}:U(\gg_{\bar 0})\to A_{\bar 0}$)  is surjective.
	 
         (b) Note that if the statement holds for $M$ then it holds for $\Pi M$.
         
             If $\gg=\mathfrak{osp}(2a+1|2b)$ then we can assume that $M$ is the pullback of a simple preferred weight $Cl(a|b)$-module $X$. By Proposition \ref{description-of-simple}(b) there is an isomorphism  $X\simeq (\left(X^+)^\dagger\otimes (X^-)^\dagger \right)^{\dagger}$  for some simple preferred weight $Cl(a|0)$-module $X^+$ and some
             simple preferred weight $D(b|0)$-module $X^-$. Next, using the isomorphism $Cl(a|b)_{\bar 0} \simeq Cl(a|0)\otimes Cl(0|b)_{ev}$ from Corollary \ref{even-iso}
             we see that
             $X_{\bar 0}\simeq X^+\otimes R(X^-)$ and $X_{\bar 1} \simeq X^+\otimes (X^-/R(X^-))$ where the functor $R$ is defined in Section 2.6.  Thus, $S = S'$ is isomorphic to the pullback to $\so (2a+1)$  of $X^+$ while $T$ and $T'$ are isomorphic to the pullbacks  to $\sp (2b)$ of
             $R(X^-)$ and $X^-/R(X^-)$,
           respectively. 

            Now let $\gg=\mathfrak{osp}(2a|2b)$. We can assume that  $M$ is the pullback of $R(X)$ for a simple preferred weight
            $Cl(a|b)_{ev}$-module $X$. Then
            $$R(X)_{\bar 0}\simeq R(X^+)\otimes  R(X^-),\quad R(X)_{\bar 1}\simeq  (R(X^+)/X^+)\otimes (X^-/R(X^-)).$$
            Therefore $S$ and $S'$ are isomorphic to the respective pullbacks to $\so (2a)$ of $R(X^+)$ and $(R(X^+)/X^+)$, and  $T$ and $T'$ are the same as in the case of $\mathfrak{osp}(2a+1|2b)$. 
            
    The uniqueness of $S$ and $T$, and hence also of $S'$ and $T'$, is clear from the isomorphism of $\mathfrak g_{\bar{0}}$-modules $ M_{\bar{0}}\simeq S\otimes T$. The fact that   $S,S',T,T'$ determine $M$ up to isomorphism is a consequence of the observation that $M_{\bar{0}}$ determines $R(X)_{\bar{0}}$, which in turn determines $X^+$ and $R(X^-)$ for $\gg=\mathfrak{osp}(2a+1|2b)$ (respectively, $R(X^+)$ and $R(X^-)$ for $\gg=\mathfrak{osp}(2a|2b)$),  and ultimately $X^+$ and $X^-$ since $R$ is an equivalence of categories. Then $M$ is the pullback of $(\left(X^+)^\dagger\otimes (X^-)^\dagger \right)^{\dagger}$  for $\gg=\mathfrak{osp}(2a+1|2b)$ and of $R(\left(X^+)^\dagger\otimes (X^-)^\dagger \right)^{\dagger}$  for $\gg=\mathfrak{osp}(2a|2b)$.

            (c) The given  pair $(S,T)$ determines a pair $(X^+, X^-)$, where $S$ is the pullback of a simple weight $Cl(a|0)$-module $X^+$ and $T$ is the pullback of a simple weight  $Cl(0|b)$-module  $(X^-)^\dagger$ for a simple weight  $D(b|0)$-module $X^{-}$. Then $M$ is recovered from $X^+$ and $X^-$ as in the proof of part (b).
	 \end{proof}
	 
	 \begin{remark}
There is an alternative definition of pairs of twins $(S,S')$ or $(T,T')$ in terms of the supports of the weight modules $S$ and $T$. Recall that in \cite{GrP} the supports of all simple bounded (equivalently, multiplicity free) weight $\so(\infty)$- and $\sp(\infty)$- modules are described explicitly, and moreover  a given such module is determined up to isomorphism by its support. For a finite-dimensional orthogonal or symplectic Lie algebra it is well known that a simple multiplicity free weight module is determined by its support as well. Both if $a< \infty$ or $a = \infty$, for any spinor $\so(2a)$-module $S$ there exists  a unique (up to isomorphism)  spinor module $S'$ such that for every $i \in \Z_{>0}$, 
$\mu + \varepsilon_i \in \supp S'$  for some $\mu \in \supp S$. Similarly, if $b< \infty$ or $b = \infty$, for every  $ \sp(2b)$-module $T$ of oscillator type there exists a unique (up to isomorphism) module $T'$ of oscillator type  such that for every $i \in \Z_{>0}$, $\nu + \varepsilon_i \in \supp S'$  for some $\nu \in \supp S$. It is straightforward to show that the pairs $(S,S')$ and $(T,T')$ are precisely the pairs of twins defined above. This observation leads to another proof of Theorem \ref{thm-m-twins}(b) based on analyzing the supports of the $\gg_{\bar{0}}$-modules $M_{\bar{0}}$ and $M_{\bar{1}}$. 
	 \hfill${\bigcirc}$    \end{remark}
	 
    Consider the decomposition $\gg_{\bar{0}} = \gg_{\mathfrak o}  \oplus \gg_{\sp} $, where $\gg_{\mathfrak o}  \simeq \so(2a)$ or $\gg_{\mathfrak o}  \simeq \so(2a+1)$ and  $\gg_{\sp}  \simeq \sp (2b)$. Set $\gh_{\so} = \gh \cap \gg_{\so}$ and  $\gh_{\sp} = \gh \cap \gg_{\sp}$. Then $\gh^* = \gh_{\so}^* \oplus \gh_{\sp}^*$. Moreover, if $\Gamma_{\so} \subset \gh_{\so}^*$  and $\Gamma_{\sp} \subset \gh_{\sp}^*$ we put  $\Gamma_{\so} + \Gamma_{\sp} := \{ \gamma_1 + \gamma_2 \; | \; \gamma_1 \in \Gamma_{\so}, \,  \gamma_2 \in \Gamma_{\sp}\}$.

\begin{corollary}
Let $M$ be as in Theorem \ref{thm-m-twins}. Then 
$$\supp M = \left( \supp S \sqcup \supp S' \right) + \left( \supp T \sqcup \supp T' \right) \subset \gh_{\so}^* \oplus \gh_{\sp}^*.$$ 
Moreover $M$ is never isomorphic to $\Pi M$, and $\supp M$ determines the isomorphism class of $M$ up to application of $\Pi$. 
\end{corollary}

\begin{remark}
The pairs $(M,\Pi M)$ for $\gg$ are appropriate superanalogs of twin pairs for $\so (2a)$ or $\sp (2b)$.
 \hfill${\bigcirc}$    \end{remark}

\section{On the category of bounded weight $\osp$-modules}
            Now  we turn our attention  to the category  $\mathcal B_{\gg}$  of bounded $\gg$-modules. In this section, $\gg$ stands for $\osp(2a+1|2b)$ or $\osp(2a|2b)$ for all, possibly finite, $a$ and $b$. 
            
            Let $\mathcal B^{osc}_{\gg}$ denote the full subcategory of $\mathcal B_{\gg}$ with simple objects 
               of spinor-oscillator type.
               Every $M\in \mathcal B_{\gg}$ decomposes uniquely into a direct sum $M'\oplus M''$ with $M'\in \mathcal B^{osc}_{\gg}$ and $M''$ being a direct sum of finitely many copies of trivial and defining modules. This follows from a simple inspection of supports which shows that any simple subquotient of $M$ isomorphic to $V, \Pi V,\C, \Pi \C$ splits as a direct summand of $M$.               By $\widetilde{\mathcal W}_A^{fin}$ for $\gg = \osp(2a+1|2b)$ (respectively, $\widetilde{\mathcal W}_{A_{ev}}^{fin}$ for  $\gg = \osp(2a|2b)$) we denote the category of all weight $A$-modules (respectively, $A_{ev}$-modules) with finite weight multiplicities. Note that the objects of $\widetilde{\mathcal W}_A^{fin}$ (respectively, $\widetilde{\mathcal W}_{A_{ev}}^{fin}$)  are not necessarily preferred $A$-modules  (respectively, $A_{ev}$-modules). 
               
               Observe that, if $a$ and $b$ are finite then the indecomposable modules in $ \mathcal B^{osc}_{\gg}$ have finite length. Indeed, the support of  every such module $M$ lies in a single coset of the root lattice of $\gg$. Since the root lattice of $\gg_{\bar 0}$ has index $2$ in the root lattice of $\gg$, the support of $M$ over $\gg_{\bar{0}}$ lies in at most $2$  cosets of the root lattice of  $\gg_{\bar{0}}$.
                 As a consequence, $M$ has finite length as a $\gg_{\bar{0}}$-module by Lemma 3.3 in \cite{Mat}.

The following is our first main result about the category $\mathcal B^{osc}_{\gg}$.

       \begin{theorem}\label{equivalence-Clifford} Let $A=Cl(a|b)$ for $b\neq 1$. If $\gg=\osp(2a+1|2b)$ then the category $\mathcal B^{osc}_{\gg}$ is equivalent to the category $\widetilde{\mathcal W}_A^{fin}$. If  $\gg=\osp(2a|2b)$ then the category $\mathcal B^{osc}_{\gg}$ is equivalent
     to  the category  $\widetilde{\mathcal W}_{A_{ev}}^{fin}$.
   \end{theorem}

        As a first step we prove Theorem \ref{equivalence-Clifford} for  finite $a$ and $b$.
        \begin{lemma}\label{lem-restriction} Let  $\dim \,\gg<\infty$. Then the restriction map  $\Ext^1_{\gg,\gh}(M,N)\to \Ext^1_{\gg_{\bar 0},\gh}(M,N)$ is injective.
        \end{lemma}
        \begin{proof} We have to show that any exact sequence in $\mathcal B^{osc}_{\gg}$
          $$0\to N\to R\to M\to 0$$
          which splits over $\gg_{\bar 0}$ splits also over $\gg$. It suffices to show that $H^1(\gg,\gg_{\bar 0}; \Hom (M,N))_{\bar 0}=0$, where $\Hom$ stands for the homomorphisms of vector spaces disregarding the $\mathbb Z_2$-grading,   see \S 3.1 and \S 4.5 of  \cite{Fuks}. 
          Any indecomposable
          object in $\mathcal B^{osc}_{\gg}$ has finite length and therefore it is enough to prove that this cohomology vanishes for simple $M$ and $N$. Writing down the first three terms of the complex computing relative cohomology, we have
          $$0\to \Hom^0_{\gg_{\bar 0}}(M,N) \xrightarrow{d} \Hom^0_{\gg_{\bar 0}}(\gg_{\bar 1}\otimes M,N)\to \Hom^0_{\gg_{\bar 0}}(\Lambda^2\gg_{\bar 1}\otimes M,N)\to\dots, $$
   where $\Hom^0_{\gg_{\bar 0}}$ denotes  homomorphisms of $\gg_{\bar{0}}$-modules preserving the $ \mathbb Z_2$-grading.     Note that the second term of the complex does not vanish if and only if $\supp M_{\bar 1}\cap(\supp N_{\bar 0}+\Delta_{\bar 1})$ or $\supp M_{\bar 0}\cap( \supp N_{\bar 1}+\Delta_{\bar 1})$ is non-empty. Using
          Theorem \ref{thm-m-twins} we see that this can happen if and only if $M\simeq N$. In the latter case
          $$\Hom^0_{\gg_{\bar 0}}(\gg_{\bar 1}\otimes M,M)=\Hom^0_{\gg_{\bar 0}}(\gg_{\bar 1}\otimes M_{\bar 0},M_{\bar 1})\oplus \Hom^0_{\gg_{\bar 0}}(\gg_{\bar 1}\otimes M_{\bar 1},M_{\bar 0})=\mathbb C^2$$
          and
          $$\End^0_{\gg_{\bar 0}}(M)=\End^0_{\gg_{\bar 0}}(M_{\bar 0})\oplus\End^0_{\gg_{\bar 0}}(M_{\bar 1})=\mathbb C^2.$$
          Consider $\varphi_0\in \Hom^0_{\gg_{\bar 0}}(\gg_{\bar 1}\otimes M_{\bar 0},M_{\bar 1})$ and
          $\varphi_1\in \Hom^0_{\gg_{\bar 0}}(\gg_{\bar 1}\otimes M_{\bar 1},M_{\bar 0})$ defined by the formula
          $\varphi_i(g\otimes m)=gm$ where $g\in\gg_{\bar 1}$ and $m\in M_i$. Set
          $$\psi_i(m):=\begin{cases}m\ \text{if}\ m\in M_i\\ 0 \ \text{if}\ m\notin M_i\end{cases}.$$
            Then $\varphi_i=d(\psi_i)$. Hence $H^1(\gg,\gg_{\bar 0}; \End (M))_{\bar 0}=0$.
          \end{proof}

             For $\gg_{\bar{0}}=\gg_{\so}\oplus\gg_{\sp}$ we say that a simple module $Z$ 
               has \emph{spinor-oscillator type} if $Z\simeq S\otimes T$ for some  spinor $\gg_{\so}$-module $S$ and
               some  $\gg_{\sp}$-module $T$ of oscillator type. By $\mathcal B^{osc}_{\gg_{\bar 0}}$ we denote the category of ${\mathbb Z}_2$-graded bounded weight
               $\gg_{\bar 0}$-modules with simple constituents of spinor-oscillator type.
          \begin{corollary}\label{fincase} Theorem  \ref{equivalence-Clifford}  holds in the case $\dim \,\gg<\infty$ (and $b\neq 1$).
          \end{corollary}
          \begin{proof} Note that if $b=0$ the statement is trivial since  $\mathcal B^{osc}_{\gg}$ is a semisimple category with objects that are 
            finite direct sums of (finite-dimensional) spinor modules. Next, for   $\gg=\mathfrak{sp}(2b)$ with $1<b<\infty$ the
            statement is proven in \cite{GrS}. Therefore,  if $\gg=\osp(2a+1|2b)$ (respectively, $\gg=\osp(2a|2b)$),  we have an equivalence of the categories
            $\mathcal B^{osc}_{\gg_{\bar 0}}$ and $\widetilde{\mathcal W}_{A_{\bar 0}}^{fin}$
          (respectively,     $\widetilde{\mathcal W}_{(A_{ev})_{\bar 0}}^{fin}$), where $\widetilde{\mathcal W}_{A_{\bar 0}}^{fin}$ (respectively, $\widetilde{\mathcal W}_{(A_{ev})_{\bar 0}}^{fin}$) is
 the category of $\mathbb Z_2$-graded weight $A_{\bar 0}$-modules (respectively, $(A_{ev})_{\bar 0}$-modules) with finite weight multiplicities.
 
            Let us prove that the pullback a projective object $P$ in  $\widetilde{\mathcal W}_A^{fin}$ (respectively,
            $\widetilde{\mathcal W}_{A_{ev}}^{fin}$) is projective in $\mathcal B^{osc}_{\gg}$. Since $P$ is induced from a finite-dimensional $H_A$-module,
            $P$ is projective in $\widetilde{\mathcal W}_{A_{\bar 0}}^{fin}$ (respectively,
            $\widetilde{\mathcal W}_{(A_{ev})_{\bar 0}}^{fin}$). By the above equivalence, the pullback of $P$ is projective in  $\mathcal B^{osc}_{\gg_{\bar 0}}$. Now
            Lemma \ref{lem-restriction}
          implies that $P$ is projective in $\mathcal B^{osc}_{\gg}$.

            Since any object $M$ in $\mathcal B^{osc}_{\gg}$ is a quotient of a projective module, $M$ is obtained by pullback from
            $\widetilde{\mathcal W}_A^{fin}$ (respectively, $\widetilde{\mathcal W}_{A_{ev}}^{fin}$)-module. The statement follows.
            \end{proof}
Next we recall the following statement.
\begin{proposition}[\cite{CP}, Corollary A.3] \label{prop-cp} Let $\gg=\displaystyle\lim_{\to}\gg_k$ be a direct limit of Lie superalgebras.
Let $Q= \displaystyle \lim_{\longrightarrow} Q_k$ and $R= \displaystyle \lim_{\longrightarrow} R_k$ be weight $\gg$-modules. Assume that $R$ has finite-dimensional weight spaces. 
Then $\Ext^1_{\gg_k,\gh_k}(Q_k, R_k) = 0$ for all $k >> 0$ implies  $\Ext^1_{\gg,\gh}(Q, R) = 0$.
 \end{proposition}

We are now ready for 

\noindent{\it Proof of Theorem  \ref{equivalence-Clifford}.} We only need to consider the case  $\gg = \osp(2a+1|2b)$ or $\gg = \osp(2a|2b)$ where
$\dim \gg = \infty$. We fix an exhaustion 
$\gg = \displaystyle \lim_{\longrightarrow}  \gg_k$ for $\gg_k = \osp (2a_k+1| 2b_k)$ or  $\gg_k = \osp (2a_k| 2b_k)$,
where one of the sequences $a_k$ or $b_k$ may stabilize. 

Since our desired equivalence will be obtained simply by pullback via the homomorphisms $\Theta_{a|b}$ or $\Psi_{a|b}$, it suffices to show that every object in  $\mathcal B^{osc}_{\gg}$
is the pullback of some weight $A$-module (respectively, $A_{ev}$-module). For this, notice that Proposition \ref{prop-cp} implies that if $P =  \displaystyle \lim_{\longrightarrow} P_k$ is a direct limit of projective objects in  $\mathcal B^{osc}_{\gg_k}$, then $P$ is projective object in  $\mathcal B^{osc}_{\gg}$. Next, Theorem \ref{equivalence-Clifford} holds for $\gg_k$ and thus every $P_k$ is the pullback of a projective object in $\widetilde{\mathcal W}_{A_k}^{fin}$ for
$\gg = \osp(2a+1|2b)$ (respectively, $\widetilde{\mathcal W}_{(A_k)_{ev}}^{fin}$ for  $\gg = \osp(2a|2b)$). Since every object of $\mathcal B^{osc}_{\gg}$ is a quotient of some $P$ as above, we conclude that every object of  $\mathcal B^{osc}_{\gg}$ is the pullback of some object of $\widetilde{\mathcal W}_{A}^{fin}$ (respectively, $\widetilde{\mathcal W}_{A_{ev}}^{fin}$).  \hfill $\square$

       \begin{corollary}
                
 Theorem         \ref{equivalence-Clifford} shows that for $b\neq 1$ any indecomposable object $M$    of $\mathcal B^{osc}_{\gg}$   has a filtration similar to the filtration which exists on an indecomposable weight $A$-module with finite-dimensional weight spaces according 
to  Theorem \ref{filtration}.    \end{corollary}         
\begin{remark} For $b=1$ every module in $\mathcal B^{osc}_{\gg}$ has finite length.
  \hfill${\bigcirc}$    \end{remark}

     The following is our second main result about the category  $\mathcal B^{osc}_{\gg}$. Let $(\mathcal B^{osc}_{\gg_{\bar{0}}})_{\bar{0}}$ be the category of purely even bounded weight $\gg_{\bar{0}}$-modules with simple constituents of spinor-oscillator type. 
                
\begin{corollary} \label{ equivalence_osp} If $b>1$ then the category $\mathcal B^{osc}_{\gg}$ is equivalent to the category $(\mathcal B^{osc}_{\gg_{\bar{0}}})_{\bar{0}}$. The functor $\mathcal{E}:M\mapsto M_{\bar{0}}$ establishes an
  equivalence.
\end{corollary}
\begin{proof} The statement follows from Theorem \ref{equivalence-Clifford} and from the equivalence of categories established in Theorem \ref{even} for $B = A_{\bar{0}}$.

\end{proof}

         \begin{corollary} For $b\geq 2$ every non-semisismple block of the category of bounded $\gg$-modules is equivalent to a block of bounded $D(k|0)$- or $D(\infty|0)$-modules with integral weights. 
         \end{corollary}

		The category of bounded weight $D(k|0)$-modules for finite $k$
		  is described, for example, in \cite{GrS}. For the case of $D(\infty|0)$ see  \cite{FGM}.

    \section {Simple bounded weight $\sl(\infty|\infty)$--modules } \label{sec-sl-bounded}

    We start by two lemmas concerning $\sl (m|n)$-modules for $m,n \in \Z_{\geq 0}$. Given a Lie superalgebra 
    $\gq \simeq \mathfrak{sl}(m|n)$ we fix the  simple roots of $\gq$ as
    $$\varepsilon_1-\varepsilon_2,\dots,\varepsilon_{m-1}-\varepsilon_m,\varepsilon_m-\delta_1,\delta_1-\delta_2,\dots,\delta_{n-1}-\delta_n,$$
    and let $\omega_1,\dots,\omega_{m-1},\omega_m,\omega_{m+1},\dots,\omega_{m+n-1}$ denote the dual basis (fundamental weights).
    There is an obvious embedding $\mathfrak{sl}(m)\subset\gq_{\bar 0}$, and  we consider $\omega_1,\dots,\omega_{m-1}$ also as fundamental weights of $\sl(m)$.

    \begin{lemma}\label{highestweight}  Let $\gq=\mathfrak{sl}(m|n)$ for   $m\geq 3$. Let 
      $M$ be a simple  bounded highest weight $\gq$-module with highest weight $\lambda$ and such that  $d(M)<m-1$. Assume that $M$ is not
      integrable over the simple ideal $\mathfrak{sl}(m)\subset\gq_{\bar 0}$. Then $\lambda=a\omega_1$ with $a\notin \mathbb Z_{\geq 0}$,
      or $\lambda=-(1+a)\omega_{k-1}+a\omega_k$ for  $2\leq k\leq m$.
   \end{lemma}
   \begin{proof}Denote by $\mu$ the weight of $\sl(m)$ obtained from $\lambda$ by restriction.
     By  Lemma \ref{finiterank}, $\Ann_{U(\mathfrak{sl}(m))}L(\mu)=\Ann_{_{U(\mathfrak{sl}(m))}}L(a\omega_1)$ or $\Ann_{U(\mathfrak{sl}(m))}L(\mu)= \Ann_{_{U(\mathfrak{sl}(m))}}L(a\omega_{m-1})$
     for some $a\notin\mathbb Z_{\geq 0}$. Since the primitive ideals $\Ann_{_{U(\mathfrak{sl}(m))}}L(a\omega_1)$ and  $\Ann_{_{U(\mathfrak{sl}(m))}}L(a\omega_{m-1})$ have degree $1$, the result of \cite{PSer3} mentioned before Lemma \ref{finiterank} shows that also $d (L(\mu)) = 1$. Therefore        Proposition 3.4 of \cite{BBL} implies that $\mu$ is one of the following
     weights:
     \begin{enumerate}
     \item $a\omega_1$ for $a\notin \mathbb Z_{\geq 0}$,
     \item$b\omega_{m-1}$ for  $b\notin \mathbb Z_{\geq 0}$,
       \item $-(1+a)\omega_{k-1}+a\omega_k$ for some   $2\leq k\leq m-1$ and arbitrary $a$.
       \end{enumerate}

       Let us deal first with the cases (1) and (3). Consider the odd reflections with respect to the roots $\varepsilon_m-\delta_1,\dots,\varepsilon_m-\delta_n$ of $\gq$.
       Since the restriction to $\mathfrak{sl}(m)$ of the highest weight of $M$ with respect to any reflected Borel subsuperalgebra must
       satisfy the same respective condition (1) or (3),
       all these reflections must be atypical. This is only possible if the restriction of $\lambda$ to the Cartan subalgebra of
       $\mathfrak{sl}(n)$ equals zero.
       Furthermore, we have $(\lambda,\varepsilon_m-\delta_1)=0$. This implies  $\lambda=a\omega_1$  or  $\lambda=-(1+a)\omega_{k-1}+a\omega_k$  for  $2\leq k\leq m-1$, respectively.

    Now let $\mu=b\omega_{m-1}$ as in (2). After performing all  odd reflections with respect to the
       roots $\varepsilon_m-\delta_1,\dots,\varepsilon_m-\delta_n$,
       we obtain a highest weight $\lambda'$ of $M$ such that its restriction to $\mathfrak{sl}(m)$ equals
       $c\omega_{m-1}$ and $b-c\in\mathbb Z_{\geq 0}$. Next, we perform odd reflections with respect to the roots
       $\varepsilon_{m-1}-\delta_1,\dots,\varepsilon_{m-1}-\delta_n$. By the same argument as in cases (1) and (3),  these latter reflections must be atypical.
       Therefore the restriction of
       $\lambda'$ to $\mathfrak{sl}(m)$ equals zero and $(\lambda',\varepsilon_{m-1}-\delta_1)=0$. In other words, $\lambda'=c\omega_{m-1}$ for some
       $c\notin \mathbb Z_{\geq 0}$. Finally,  passing via odd reflections to the original
       Borel subsuperalgebra yields $\lambda=b\omega_{m-1}+(1-b)\omega_{m}$. To finish the proof we set $a=1-b$.
        \end{proof}
   
  Recall the homomorphisms $\Upsilon^+_{m|n} : U(\sl (m|n))\to D(m|n)_{{0}}$ and $\Upsilon^-_{m|n} : U(\sl (m|n))\to D(n|m)_{{0}}$ from Section \ref{sebsec-conn}. Note that those homomorphims map the Cartan algebra of $\sl(m|n)$ to the subalgebra spanned by
        $u_i-u_j$ for all $i,j\neq 0, i\neq j$. 
          Moreover, the map $f$ induced by  $\Upsilon^+_{m|n} $  (respectively, $\Upsilon^-_{m|n}$) from 
          $\left(\mbox{Span}\, \{ u_i-u_j \; | \; i \neq j\}\right)^*$ to $\mathfrak h^*$ is linear, and is determined by the correspondence
             $\zeta_i \mapsto \varepsilon_i $, $\zeta_{-j} \mapsto \delta_j$
        (respectively, $\zeta_{-i} \mapsto \varepsilon_i$, $\zeta_{j} \mapsto \delta_j$).

   \begin{corollary}\label{impcor} Let $M$ be a bounded simple non-integrable $\gq=\mathfrak{sl}(m|n)$-module with  $d(M)<\min(m,n)-1$. Then $d=1$ and $\Ann_{U(\gq)} M$
     contains $\operatorname{ker}\Upsilon^-_{m|n}$ or $\operatorname{ker}\Upsilon^+_{m|n}$.
   \end{corollary}
   \begin{proof} Without loss of generality we can assume that $M$ is not integrable over $\mathfrak{sl}(m)$. Then 
     $\Ann_{U(\gq)} M=\Ann_{U(\gq)} L(\lambda)$ where $\lambda$ is one of the weights in Lemma \ref{highestweight}. It suffices to show that $L(\lambda)$ is obtained
     by pullback from a weight $D(m|n)_0$-module. Consider the $D(m|n)$-module 
     $$F (\lambda):=x_k^a\mathbb C[x_1^{\pm 1},\dots,x_m^{\pm 1},x_{-1},\dots,x_{-n}],$$
     where $k=1, a \notin \Z_{\geq 0}$  if $\lambda = a\omega_1$, and $a \in \C$ if  $\lambda = -(1+a)\omega_{k-1}+a\omega_k$ for $2\leq k\leq m$.

     Let $f_k:=x_1^{-1}\dots x_{k-1}^{-1}x_k^a$ and let $Y_k$ denote the $D(m|n)_0$-submodule in $F  (\lambda)$ generated by $f_k$. Note that the weight of $f_k$
     equals
     $a\omega_1$ for $k=1$, and equals $-(1+a)\omega_{k-1}+a\omega_k$ for $2\leq k\leq m$. Then $Y_1$ is a simple $D(m|n)_0$-module and its
     pullback along $\Upsilon^+_{m|n}$ is isomorphic to $L(a\omega_1)$,  since by direct computation one can see that any vector annihilated by all $x_{i-1}\partial_{i}$ for  $2\leq i\leq m$ is proportional to $f_1$. For $k>1$, consider the $D(m|n)_0$-submodule $Z_k\subset Y_k$ 
     generated by $x_{i-1}\partial_{i}(f)$ for  $2\leq i\leq m$. Then $f\notin Z_k$ and hence $X_k:=Y_k/Z_k\neq 0$. Furthermore, the pullback
     along $\Upsilon^+_{m|n}$ of $X_k$ is
     isomorphic to $L(-(1+a)\omega_{k-1}+a\omega_k)$ (again because any vector annihilated by all $x_{i-1}\partial_{i}$ is proportional to $f_k$).
     The statement is proved.
     \end{proof} 

In the rest of this section $\gg =\sl (\infty | \infty)$ and $A = D(\infty | \infty)$. We fix an exhaustion $\gg = \limarr \gg_k$, where $\gg_k \simeq \sl(k|k)$. By $\gg^{\pm}$ we denote the ideals of $\gg_{\bar{0}}$ with respective roots $\varepsilon_i - \varepsilon_j$ and $\delta_i - \delta_j$, and we write $\Upsilon^\pm$ instead of $\Upsilon^\pm_{\infty | \infty}$.

     \begin{lemma}\label{limitsl} Let $M$ be a simple bounded weight $\gg$-module not integrable over $\gg^+$ or  $\gg^-$. Then
       $\Ann_{U(\gg)} M$ contains  $\operatorname{ker}\Upsilon^+$, or respectively $\operatorname{ker}\Upsilon^-$,  and therefore $M$ is multiplicity free.
          \end{lemma}
          \begin{proof} Let $v\in M$ be a nonzero weight vector and let $M_k:=U(\gg_k)v$. If $k>d(M)$ then
            Corollary \ref{impcor} implies that $\Ann_{U(\gg_k)} M_k$ contains $\operatorname{ker}\Upsilon^\pm_{k|k}$. Therefore
            $\Ann_{U(\gg)} M=\limarr \Ann_{U(\gg_k)} M_k$ contains $\operatorname{ker}\Upsilon^\pm=\limarr \operatorname{ker}\Upsilon^\pm_k $. Since every simple weight
            $A_0$-module  is multiplicity free, the second assertion follows. 
          \end{proof}

          \begin{remark} One may observe that Lemma \ref{limitsl} holds also for the Lie superalgebra $\mathfrak{sl}(\infty|n)$, $n\in\mathbb Z_{>0}$,
            where one replaces $\gg^+$ by the simple ideal $\sl(\infty)$ of  $\mathfrak{sl}(\infty|n)_{\bar 0}$ and $\Upsilon^+$ by  $\Upsilon^+_{\infty|n}$.
            \hfill${\bigcirc}$    \end{remark}
     

The simple bounded integrable $\gg$-modules have been classified in \cite{CP}, Theorem 5.9. Therefore, in order to classify all simple bounded weight $\gg$-modules it suffices to prove the following. 

          \begin{theorem}\label{thm:sl-simple} Let $M$ be a simple bounded non-integrable $\gg$-module. Then
              \begin{enumerate}
              \item[(a)] $M$ is multiplicity free.
              \item[(b)] $M$ is obtained from a simple weight $A_0$-module by pullback via precisely one of the homomorphisms
                $\Upsilon^+$ or $\Upsilon^-$, and accordingly either $\gg^-$ or $\gg^+$ acts integrably on $M$.
              \item[(c)] Pullback via $\Upsilon^\pm$  establishes a bijection between isomorphisms classes of simple, non-integrable over
                $\gg^\pm$, bounded
                $\gg$-modules and isomorphism classes of simple non-integrable weight $A_0$-modules.
                \end{enumerate}
          
              \end{theorem}
              \begin{proof} (a) follows directly from Lemma \ref{limitsl}.
              
(b)               Let $C^{\pm}$ denote the image of $\Upsilon^{\pm}$. It is easy to see that $C^{\pm} \subsetneq A_0$. Lemma \ref{limitsl} implies that every simple non-integrable weight $\gg$-module 
              is obtained by
                pullback from a simple $C^+$- or a $C^-$-module. Therefore, to prove (2) we need to show that a weight $\gg$-module obtained by pullback from a weight $C^{\pm}$-module is in fact obtained by pullback from the restriction of a weight $A_0$-module to $C^{\pm}$.
                It suffices to prove the statement for $C^+$, since the other case  follows by applying the obvious automorphism of $\gg$.
                
                Recall the basis $\{u_i\}_{i\in\mathbb Z}$ of $\gh_A$ introduced in Section 2.4. By a slight abuse of notation we denote by the same letter
                the preimage of $u_i$ in the Cartan subalgebra of $\mathfrak{gl}(\infty|\infty)$. Then $\{w_i=u_i-u_{-1}\mid i\neq -1\}$
                is a basis of the Cartan subalgebra of $\gg$. Let $N$ be a simple weight $\gg$-module, $\mu\in \supp N$ and $c\in\mathbb C$. Note that we can endow $N$ with a  $\mathfrak{gl}(\infty|\infty)$-module structure by setting      $u_{-1}v:=(c+\nu(u_{-1})-\mu(u_{-1}))v$ for every $v\in N^{\nu}$. We denote this  $\mathfrak{gl}(\infty|\infty)$-module by $N(\mu,c)$.

                We claim that if $M$ is the pullback of some simple weight  $C^+$-module, then we can find $\mu,c$ such that the
                $\mathfrak{gl}(\infty|\infty)$-module
                $N(\mu,c)$ is the pullback of some weight $A_0$-module. Clearly, we can assume that $M$ is not trivial. We pick some
                $\kappa\in \supp M$ such that $\kappa(w_i)\neq 0$ for some $i\leq -2$.
                One readily sees that the relation $w_i^3=w_i$ implies $\kappa(w_i)=\pm 1,0$ for $i\leq -2$. Next, we choose a negative $i$ such that
                $\kappa(w_i)\neq 0$. It easily follows from the linearity of $\kappa$ that $\kappa(w_j)=0$ or  $\kappa(w_j)=\kappa(w_i)$ for every negative $j$.
                Finally, we set $c=0$ if $\kappa(w_i)=1$ and $c=1$ if $\kappa(w_i)=-1$. Then $\supp M(\kappa,c)\subset \gh^\vee_{A}$ and, since the restriction
                of $M(\kappa,c)$ to $\gg$ is the pullback of some weight $C^+$-module, the $\mathfrak{gl}(\infty|\infty)$-module $M(\kappa,c)$ is the pullback
                of a weight $A_0$-module.

                (c) Follows from Proposition \ref{integrable}(b). 
                 \end{proof}
                 \begin{remark} It is  likely that Theorem \ref{thm:sl-simple} holds also for $\mathfrak{sl}(\infty|n)$.
                   \hfill${\bigcirc}$    \end{remark}
     
          \begin{remark} \label{rem-theta} Note that the definition of $\gh_A^{\vee}$ implies that if $M$ is the pullback of a weight $A_0$-module via $\Upsilon^+$ (respectively,  $\Upsilon^-$), then for $\sum_i a_i \varepsilon_i + \sum_j b_j \delta_j  \in \supp M$ we have $a_i \in \{ 0,1\}$  (respectively,  $b_j \in \{ 0,1\}$).
                   \hfill${\bigcirc}$    \end{remark}
     
      \begin{proposition}\label{supportisom} A simple bounded weight $\gg$-module $M$ is determined, up to isomorphism and a
        possible parity change, by $\supp M$.
      \end{proposition}
      \begin{proof} Here we consider the case of non-integrable modules, and leave as an exercise to the reader to check our claim for integrable modules using the classification result 
      of \cite{CP}.
        Let us observe  that if $M$ and $N$ are not integrable, and one is obtained by pullback via $\Upsilon^+$ while the other is obtained by pullback via $\Upsilon^-$, then $M$ and $N$ cannot  have the same support.

Now, without loss of generality we can assume that $M$ and $N$ are obtained by pullback from simple weight $A_0$-modules $X$ and $Y$, respectively.
        Suppose that $\supp M=\supp N$ but $\supp X\neq\supp Y$. Then $\supp X=\supp Y\pm\tau$ where $\tau= \sum_{i>0}(\varepsilon_i-\delta_i)$.
        Since the supports of $X$ and $Y$ are subsets of $\gh_A^\vee$, this is only possible  if  $\supp X=\{0\}$, $\supp Y=\{\pm \tau\}$ or vice versa.
        Then, both $M$ and $N$ are necessarily trivial and we have a contradiction. Consequently, $\supp M=\supp N$ implies $\supp X=\supp Y$, and then the     $A_0$-modules $X$ and $Y$ are isomorphic up to parity change
     by Theorem \ref{Zero}(a). This completes the proof.
   \end{proof}

          Let $M^\pm(\mu)$ denote the simple weight $\gg$-module obtained by pullback from  the simple  weight $A_0$-module $Y(\mu)$ via $\Upsilon^\pm$.
          
    \begin{proposition}\label{intupsilon} Every multiplicity free simple weight $\gg$-module $M$ is isomorphic to the
        pullback of a
        simple weight $A_0$-module via $\Upsilon^+$ or $\Upsilon^-$. If $M$ is obtained by pullback via both $\Upsilon^+$ and $\Upsilon^-$, then $M$ is
        isomorphic to $V$, $\Pi V$, $V_*$, $\Pi V_*$, $\C$ or $\Pi \C$.
      \end{proposition}
      \begin{proof} By Theorem \ref{thm:sl-simple} all non-integrable simple bounded $\gg$-modules are pullbacks of $A_0$-modules via $\Upsilon^+$ or $\Upsilon^-$, and hence are multiplicity free. Therefore it suffices to check the
        statement for
        integrable multiplicity free modules. 
        
     Theorem 5.9 in  \cite{CP} implies that, in addition to the six modules $V$, $\Pi V$, $V_*$, $\Pi V_*$, $\C$, $\Pi \C$ there are four families of multiplicity free simple integrable $\gg$-modules $S_{\mathcal A}^{\infty} V$, $S_{\mathcal A}^{\infty} V_*$, $\Lambda_{\mathcal A}^{\infty} V$, $\Lambda_{\mathcal A}^{\infty} V_*$. If one observes that all other three families are obtained from  $S_{\mathcal A}^{\infty} V$ by a twist from a proper automorphism of $\sl (\infty |\infty)$, it remains to check that any simple module of the form  $S_{\mathcal A}^{\infty} V$ is isomorphic to $M^-(\mu_{\mathcal A})$ or  $\Pi M^-(\mu_{\mathcal A})$  for a weight $\mu_{\mathcal A} \in \gh_A^{\vee}$.

Recall from \cite{CP} that $\mathcal A$ is a sequence of pairs $(a_n,b_n)$ where $a_1 \leq a_2 \leq \cdots$ is a sequence of positive integers and $b_n \in \{0,1 \}$ with the condition that $b_n=b_{n+1}$ if    $a_n=a_{n+1}$. Moreover,  $S_{\mathcal A}^{\infty} V$ is defined as the direct limit $\limarr \Pi^{b_n} S^{a_n} V_n$ where $V_n$ is the natural $\sl (n|n)$-module. Let 
$$
\mu_{\mathcal A} := \sum_{i>0} \left( b_i\varepsilon_i  + (a_i - a_{i-1} - b_i)\delta_i \right),
$$
where we set $a_0 = 0$. Then a direct verification shows that $\supp S_{\mathcal A}^{\infty} V = \supp M^-(\mu_{\mathcal A})$.       Since a simple multiplicity free weight $\gg$-module is determined by its support up to isomorphism and a possible application by $\Pi$, we conclude that $S_{\mathcal A}^{\infty} V \simeq  M^-(\mu_{\mathcal A})$ if the weight space $ \left(S_{\mathcal A}^{\infty} V\right)^{\mu_{\mathcal A}} $ has parity equal to $p(\mu_{\mathcal A})$, and  $S_{\mathcal A}^{\infty} V \simeq \Pi  M^-(\mu_{\mathcal A})$ otherwise. In fact, the parity of the weight space  $ \left(S_{\mathcal A}^{\infty} V\right)^{\mu_{\mathcal A}} $ depends only on $b_1$: the weight space $ \left(S_{\mathcal A}^{\infty} V\right)^{\mu_{\mathcal A}} $  is purely even for $b_1 =0$ and purely odd for $b_1 =1$.

Finally, the fact that each of the six modules $V$, $\Pi V$, $V_*$, $\Pi V_*$, $\C$, $\Pi \C$ is obtained by pullback  via both $\Upsilon^+$ and $\Upsilon^-$ is straightforward.

        \end{proof}

      \begin{proposition} \label{restriction}  If $M$ is a simple bounded weight $\gg$-module then $M$ is semisimple as
      a  $\gg_{\bar 0}$-module.
      \end{proposition}
      \begin{proof} The statement is clear for integrable modules since every bounded integrable $\gg_{\bar 0}$-module is semisimple by Theorem 3.7 in \cite{PSer2}.
        Therefore, without loss of generality we can assume that $M$ is isomorphic to $M^{\pm}(\mu)$.  Consider the lattice $Q_{(A_0)_{\bar{0}}}$ with generators $\varepsilon_i -\varepsilon_j, \delta_i - \delta_j$. Set
        $$Y(\mu)^n:=\bigoplus_{\nu\in \mu+n(\varepsilon_1-\delta_1) + Q_{(A_0)_{\bar{0}}}}Y(\mu)^\nu.$$
        Then $Y(\mu)^n$ is a simple $(A_0)_{\bar 0}$-module and $Y(\mu)=\bigoplus_{n\in\mathbb Z}Y(\mu)^n$. Obviously, the semisimplicity of $Y(\nu)$ over $(A_0)_{\bar 0}$ implies semisimplicity of $M$ over $\gg_{\bar 0}$.
        
        \end{proof}
 
        \begin{theorem}\label{primitivesl}
          (a)  Let $\gg=\mathfrak{sl}(\infty)$. The following ideals are all bounded primitive  ideals of $U(\gg)$:
          $$\Ann_{U(\gg)} \mathbb S_\lambda V ,\;  \Ann_{U(\gg)} \mathbb S_\lambda V_*, \; \ker\Upsilon^+, \; \ker \Upsilon^-.$$

          (b) Let $\gg=\mathfrak{osp}(2a+1|2b)$ (respectively, $\gg=\mathfrak{osp}(2a|2b)$)  with at  least one of $a$ and $b$ equal infinity. Then
          $U(\gg)$ has exactly three bounded        primitive ideals: the augmentation ideal, $\Ann_{U(\gg)} V$, and $\ker\Theta_{a|b}$ (respectively, $\ker\Psi_{a|b}$).
          \end{theorem}
          \begin{proof}  (a) follows from Corollary \ref{primcor}, Theorem \ref{thm:sl-simple}, Proposition \ref{supportisom}, and the classification of simple bounded integrable $\gg$-modules in \cite{CP}.
            (b) follows from Corollary \ref{cor-osc}.
            \end{proof}

          \begin{lemma}\label{extensions} Let $\mathcal B$ denote the category of bounded weight $\gg$-modules, and let $M,N \in \mathcal B$. Denote by $Q_{\gg}$ the root lattice of $\gg$.
            \begin{enumerate}
            \item[(a)]  $\Ext_{\mathcal B}^1(M,N)=\Ext_{\mathcal B}^1(N,M)$.
              \item[(b)] If $M$ is simple and $\Ext_{\mathcal B}^1(M,N)\neq 0$, then $\supp M\subset \supp N+Q_{\gg}$.
             \item[(c)] $\Ext_{\mathcal B}^1(M^{\pm} (\mu),\Pi M^{\pm} (\nu)) =  0$ for all $\mu,\nu$.
            \item[(d)] If $\Ext_{\mathcal B}^1(M^+(\mu),M^+(\nu))\neq 0$, then $\mu-\nu\in Q_{A_0}$ or at least one of $M^+(\mu)$ and $M^-(\nu)$ is trivial.
            \item[(e)] If $\Ext_{\mathcal B}^1(M^+(\mu),M^-(\nu))\neq 0$, then at least one of $M^+(\mu)$ and $M^-(\nu)$ is isomorphic to $V$, $V_*$, $\C$.
              \item[(f)] If $M$ and $N$ are simple and $d(M)>1$ then $\Ext_{\mathcal B}^1(M,N)=0$.              \end{enumerate}           
          \end{lemma}
          \begin{proof} (a) We consider the (contravariant) functor of contragredient duality ${\cdot}^\vee$ on the category $\mathcal B$.
            Then $M^\vee\simeq M$, $N^\vee\simeq N$ and
            $$\Ext_{\mathcal B}^1(M,N)=\Ext_{\mathcal B}^1(N^\vee,M^\vee)=\Ext_{\mathcal B}^1(N,M).$$

            (b) Let $0 \to M \to R \to N \to 0$ represent a nonzero element of $\Ext_{\mathcal B}^1(M,N)$. Then, for some weight vector $v \in N$, the image of $M$ in $R$ is a submodule of $U(\gg)v'$ where $v'$ is a preimage of $v$ in $R$ of weight $\kappa$. Then $\supp M\subset  \kappa +Q_{\gg} \subset \supp N+Q_{\gg}$.

(c) follows from comparing the parity of weight spaces of the modules $M^{\pm} (\mu)$ with the parity of the weight spaces of the modules  $\Pi M^{\pm} (\nu)$.

            (d) follows from (b).

            (e)            For a
            $\gg$-module $M$ and a Lie subsuperalgebra $\gk$ of $\gg$ we denote by 
            $\Gamma_{\gk}M$ the set of locally finite $\gk$-vectors, i.e.,
            $$\Gamma_{\gk}M :=  \{ m \in M \; | \; \dim \, \mbox{span} \{m, km, k^2m,...\} <\infty \; \forall k \in \gk \}.
            $$
             The superspace $\Gamma_{\gk}M$ is a $\gg$-submodule of $M$. This is established for Lie algebras in particular in Theorem 8.2 in \cite{PH}, and the proof for Lie superalgebras is the same.
             
            By the semisimplicity result in \cite{CP}, at least one of $M^+(\mu)$ and $M^-(\nu)$ can be assumed non-integrable. Moreover, by (a), the statement is symmetric with respect to $M^+(\mu)$ and $M^-(\nu)$. 
                      Without loss of generality, assume that $M^-(\nu)$ is not integrable. Consider a non-split exact sequence
                      $$0\to M^-(\nu)\to N\to M^+(\mu)\to 0.$$
                                      Since $\Gamma_{\gg^+}M^-(\nu)=M^-(\nu)$, there exists a root $\alpha$ of $\gg^-$ such that $\gg_{\alpha}$ acts freely on $M^-(\nu)$. If  $\Gamma_{\gg_{\alpha}}N\neq 0$ then
                                       $\Gamma_{\gg_{\alpha}}N$ is a submodule of $N$ which does not coincide with $M^-(\nu) $, i.e., the  sequence splits.  Consequently, $\Gamma_{\gg_{\alpha}}N=0$. Hence, for any
                      $\theta\in\supp M^+(\mu)$ we have $\theta+n\alpha\in\supp M^{-}(\nu)$ for  $n\geq 2$. Thus, we get that $\theta_i\in\{0,1\}$ for all $i>0$, and therefore for all $i$ by Remark \ref{rem-theta}. This is possible if and only if
                      $M^-(\nu)$ is isomorphic to $V$, $V_*$ or $\C$.

            (f) If $d(M)>1$ then using Theorem 5.9 in \cite{CP} and Proposition \ref{finrankintegrablesl} one can verify that 
            $M$ is isomorphic to $\mathbb S_\lambda V $, $\Pi \mathbb S_\lambda V $,  $\mathbb S_\lambda V_*$, or $\Pi \mathbb S_\lambda V_*$ for some Young diagram $\lambda$ with more than one row or more than one column.
            Assume 
            $M=\mathbb  S_\lambda V $ and $\Ext^1(M,N)\neq 0$. Then the semisimplicity result  (Theorem 6.1) in \cite{CP}  implies that $N$ is not integrable.
            Consider a non-split exact sequence $$0\to N\to R\to M\to 0.$$ Suppose  $N \simeq M^- (\nu)$
            for some $\nu$. The argument in the proof of (e) can be easily modified to show that  $(\theta,\alpha^\vee)\in \{\pm 1,0\}$ for any weight $\theta$ of $M$ and any root $\alpha$ of $\gg^+$.
            This implies that $\lambda$ consists of a single column, and hence
            $d(M)=1$. Similarly, if $N\simeq M^+(\nu)$ one proves that $\lambda$ consists of a single row and $d(M)=1$.
           \end{proof}

   \end{document}